\documentclass[12pt]{article}

\usepackage[T2A]{fontenc}
\usepackage[utf8]{inputenc}
\usepackage[russian,ukrainian,english]{babel}

\usepackage{amssymb}
\usepackage{graphics}
\usepackage{mathtools} 

\newcommand{\qed}{$\;\;\;\Box$}
\newenvironment{proof}{\par\smallbreak{\sl Proof.~}}
{\unskip\nobreak\hfill \qed \par\medbreak}

\newcounter{claim}
\renewcommand{\theclaim}{\arabic{claim}}
\newenvironment{claim}{\refstepcounter{claim}%
	\par\medskip\par\noindent{\bf Claim~\theclaim.}\rm}%
{\par\medskip\par}

\newenvironment{subproof}{\par\noindent{\bf Proof of Claim.}}%
{\qed\par\smallbreak}
\newcommand{\hide}[1]{}

\newcommand{\bbbn}{{\mathbb N}}

\newcommand{\bbbr}{{\mathbb R}}

\newcommand{\Z}{{\mathbb Z}}
\newcommand{\bbbm}{{\mathbb M}}

\newcommand{\beq}{\begin{equation}}
	\newcommand{\ee}{\end{equation}}

\renewcommand{\d}{\partial}

\newtheorem{thm}{Theorem}[section]
\newtheorem{lemma}[thm]{Lemma}

\newtheorem{defn}[thm]{Definition}

\newtheorem{ex}[thm]{Example}

\newcommand{\reff}[1]{(\ref{#1})}

\newcommand{\diag}{\mathop{\rm diag}\nolimits}

\newcommand{\dd}{\!\;\mathrm{d}}

\title{Lyapunov function and smooth periodic solutions to 
	quasilinear 1D hyperbolic systems
}

\newcounter{thesame}
\author{
	Irina Kmit
	\thanks{Institute of Mathematics, Humboldt University of Berlin. On leave from the
		Institute for Applied Problems of Mechanics and Mathematics,
		Ukrainian National Academy of Sciences, Lviv. {\small   E-mail:
			{\tt irina.kmit@hu-berlin.de}}}
	\ \ \ Viktor Tkachenko \thanks{Institute of Mathematics,
		National Academy of Sciences of	Ukraine, Kyiv.
		{\small   E-mail:
			{\tt vitk@imath.kiev.ua}}
}}

\date{}

\begin{document}
	
	\maketitle

\begin{abstract}
We apply a Lyapunov function to obtain conditions for the existence and uniqueness of small classical
time-periodic solutions to first order quasilinear 1D hyperbolic systems with (nonlinear) nonlocal boundary conditions  in a strip. 
The boundary conditions cover different types of reflections from the boundary as well as integral operators with delays.
In the first step we use a Lyapunov approach to derive  sufficient conditions for the robust exponential stability 
of the  
boundary value problems for a linear(ized)  homogeneous problem. Under  those conditions and a number of
non-resonance conditions, 
 in the second step we  prove the existence and
uniqueness of smooth time-periodic solutions to the corresponding linear nonhomogeneous problems. In the third step, we prove a
perturbation theorem stating that the periodic solutions survive under small perturbations of all coefficients
of the hyperbolic system. In the last step, we apply the linear results to construct small and smooth time-periodic solutions to the quasilinear problems.  
\end{abstract}

\section{Introduction}

\subsection{Problem setting and main result} 
In the domain $\Pi = \{(x,t)\in  \bbbr^2\,:\,0\le x\le1\}$ we consider
first order quasilinear hyperbolic systems of the type
\begin{equation}\label{1}
\partial_t u + A(x,t,u)\partial_x u =  F(x,t, u),
	\quad x\in(0,1), 
\end{equation}
where  $u=(u_1,\ldots,u_n)$ and $F = (F_1,\dots,F_n)$ are  vectors of real-valued functions,
$A=\diag(A_1,\dots,A_n)$ is a matrix with real-valued entries, and $n\ge 2$
is an integer.
Suppose that in some neighbohrhood of $u = 0$ in $ \bbbr^n$ we have
\begin{eqnarray}\label{strict}
A_1(x,t, u)>\dots>A_m(x,t,u)> 0>A_{m+1}(x,t,u)>\dots>A_n(x,t,u)
\end{eqnarray}
for all $(x,t)\in\Pi$ and an  integer $m$ in the range $0\le m\le n$. 

We supplement the system (\ref{1})  with the boundary conditions 
\begin{eqnarray}\label{2}
\begin{array}{ll}
	u_{j}(0,t)= R_ju(\cdot,t) + H_j\left(t, Q_j(t)u(\cdot,t)\right), \quad 1\le j\le m,\\ [1mm]
	u_{j}(1,t)= R_ju(\cdot,t),+ H_j\left(t, Q_j(t)u(\cdot,t)\right), \quad m< j\le n.
\end{array}
\end{eqnarray}
Here $R_j$ are linear bounded operators from $C([0,1];\bbbr^n)$ to $\bbbr$  defined  by
\begin{eqnarray}
	\displaystyle R_jv=\sum\limits_{k=m+1}^nr_{jk} v_k(0)+\sum\limits_{k=1}^mr_{jk} v_k(1),\quad  j\le n,
	\label{eq:R} 
\end{eqnarray}
$r_{jk}$ are real constants, and $H_j(t,v)$
are  real valued functions from $ \bbbr\times\bbbr$ to~$ \bbbr$. Moreover,
for any $t\in\bbbr$, 
$ Q_j(t)$ are linear bounded operators from $C([0,1];\bbbr^n)$ to $\bbbr$. 
Write 
$$
\begin{array}{ll}
R=(R_1,\dots,R_n),\quad Q(t)=\left(Q_1(t),\dots,Q_n(t)\right),\\ [2mm]
H(t,Q(t)v) = (H_1(t,Q_1(t)v),\dots,H_n(t,Q_n(t)v)).
\end{array}
$$

The purpose of the paper is to establish conditions 
for the existence and
uniqueness of small classical  
time-periodic solutions   to the problem  (\ref{1}), (\ref{2}). 
Since the nonlinearities $H_j$
in (\ref{2}) will be treated by a perturbation argument, 
using the concept of  a Lyapunov function,  we first establish conditions 
for the robust exponential stability of the linearized homogeneous problem without the nonlinearities in (\ref{2}) (when $H_j\equiv 0$ for all $j\le n$). Next we prove the 
existence and
uniqueness of classical  
time-periodic solutions   to the  nonhomogeneous linear problem, under the conditions of the robust exponential stability
and a number of nonresonance conditions. 
Based on this result and its proof, we then prove 
a perturbation theorem stating that  the time periodic solutions to the linearized problem survive under small perturbations of the coefficients of the linear hyperbolic system.
Finally, we apply the obtained results for constructing  the classical  time-periodic solutions to the quasilinear problem.

 We will denote by $\langle \cdot,\cdot\rangle $  the scalar product in $\bbbr^n$.
 We will use the same notation $\|\cdot\|$ for norms of vectors and matrices of different sizes. Thus, if $A$ is
an $k$-vector, then $\|A\|=\max_{j\le k}|A_j|$. If $A$ is
an $k\times k$-matrix, then $\|A\|=\max_{i,j\le k}|A_{ij}|$. By $\bbbm_k$ we denote the space of  $k\times k$-matrices
with real entries.

For a given subdomain $\Omega$ of $\Pi$ or  $\bbbr$,
let $BC(\Omega;\bbbr^n)$
be the Banach space of
all bounded and continuous maps
$u:\Omega \to \bbbr^n$
with the $\sup$-norm
$$
\|u\|_{BC(\Omega;\bbbr^n)}=\sup\left\{\|u(z)\|\,:\,z\in\Omega\right\}.
$$
Analogously one introduces the spaces $BC^k_t(\Omega;\bbbr^n)$ of functions 
$u \in BC(\Omega;\bbbr^n)$ such that $\partial_t u, \dots, \partial_t^k u \in BC(\Omega;\bbbr^n),$ with norm
$$\|u\|_{BC_t^k(\Omega;\bbbr^n)}= \sum_{j=0}^k  \|\partial_t^j u\|_{BC(\Omega;\bbbr^n)}.$$
We also use
the spaces $BC^k(\Omega;\bbbr^n), k =1,2$, of bounded and $k$-times 
continuously differentiable functions. 
 
 Let $T>0$ be a fixed positive real.  Denote by $C_{per}(\Omega;\bbbr^n)$   the vector space of all continuous maps $u : \Omega\to\bbbr^n$
which are $T$-periodic in $t$, endowed with the $BC$-norm.
For $k=0,1,2$, we also  introduce  spaces 
$$
\begin{array}{ll}
C_{per}^k(\Omega;\bbbr^n)=BC^k(\Omega;\bbbr^n)\cap C_{per}(\Omega;\bbbr^n)
\end{array}
$$
and write $C_{per}(\Omega;\bbbr^n)$ for  $C_{per}^0(\Omega;\bbbr^n)$.

For the space of linear
bounded operators from $X$ to $Y$ we use the usual notation ${\cal L}(X,Y)$, and write ${\cal L}(X)$ for ${\cal L}(X,X)$.
The Sobolev space of all
vector-functions from  $L^2\left((0,1),\bbbr^n\right)$ whose distributional derivatives 
belong to the same space $L^2\left((0,1),\bbbr^n\right)$  is denoted, as usual, by  $H^1\left((0,1),\bbbr^n\right)$.

Suppose that the data of the problem  (\ref{1}), (\ref{2})
satisfy the following conditions.
\begin{itemize}
	\item[\bf(A1)]
 There exists $\delta_0 > 0$  such that
\begin{itemize}
	\item 
 the coefficients $A(x,t,v)$ and $F(x,t,v)$ 
are $T$-periodic in $t$ and have bounded and
continuous partial derivatives  up to the second order in $x\in [0,1]$, $t\in\bbbr$,  and $v\in \bbbr^n$
with $\|v\| \le \delta_0$,
\item
 there exists $a_0>0$ such that for all  $\|v\| \le \delta_0$ it holds
\begin{equation}\label{hyp}
	\begin{array}{lll}
		\inf\left\{A_j(x,t,v)\,: \,(x,t)\in\Pi,\, 1\le j\le m\right\}\ge a_0,\\ [1mm]
		\sup\left\{A_j(x,t,v)\,:\, (x,t)\in\Pi,\,m+1\le j\le n\right\}\le -a_0,\\ [1mm]
		\inf\left\{|A_j(x,t,v)-A_k(x,t,v)|\,:\, (x,t)\in\Pi,\, 1\le j\ne k\le n\right\}\ge a_0.
	\end{array}
\end{equation}
\end{itemize}

\item[\bf(A2)] 
  $Q(t)$ is  a a one-parameter family (where $t\in \bbbr$ serves as the parameter) of homogeneous bounded linear operators mapping $C([0,1];\bbbr^n)$ to $\bbbr^n$.
  Moreover, the family  $Q(t): C([0,1];\bbbr^n)\to \bbbr^n$ is $T$-periodic and two-times continuously differentiable in $t$.
Furthermore, for all $v\in 
BC^2_t(\Pi;\bbbr^n)$ and $j\le n$, it holds
\begin{equation}\label{R'}
\begin{array}{rcl}
	\displaystyle  \frac{d}{dt}[Q_j(t)v(\cdot,t)]&=& \displaystyle Q_j^\prime(t) v(\cdot,t)
	+ \widetilde Q_j(t) \partial_tv(\cdot,t),\\ [2mm]
	\displaystyle \frac{d^2}{dt^2}[Q_j(t)v(\cdot,t)]&=& \displaystyle Q_{0j}(t) v(\cdot,t) +Q_{1j}(t) \partial_tv(\cdot,t)
+ Q_{2j}(t) \partial_t^2v(\cdot,t),
\end{array}
\end{equation}
where $Q_j^\prime(t),\widetilde Q_j(t), Q_{0j}(t), Q_{1j}(t) ,Q_{2j}(t)$ are certain  one-parameter families of homogeneous bounded linear operators
mapping $C([0,1];\bbbr^n)$ to $\bbbr$. Moreover, these families are continuous and 
 $T$-periodic  in $t$.

\item[\bf(A3)]
 For all  $j\le n$,
 the function $H_j(t,v)$   is $T$-periodic in $t$ and have 
continuous partial derivatives  up to the second order in  $t \in\bbbr$ and $v\in\bbbr$.

\end{itemize}

To formulate our main result, let us consider the problem (\ref{1}), (\ref{2})  with $H_j\equiv 0$ for all $j\le n$.
The homogeneous part of a linearized version of it at $u=0$ reads as follows:
\begin{equation}\label{1lin}
	\partial_tu  + a(x,t)\partial_x u + b(x,t) u = 0, \quad x\in(0,1),
\end{equation}
\begin{equation}\label{2lin}
\begin{array}{ll}
	u_{j}(0,t)= R_ju(\cdot,t), \quad 1\le j\le m,\\ [1mm]
	u_{j}(1,t)=R_ju(\cdot,t), \quad m< j\le n,
\end{array}
\end{equation}
where $a(x,t)=A(x,t,0)$, $b(x,t)= \partial_3 F(x,t,0)$, and
 $\partial_j$ here and in what follows denotes the partial derivative with respect to the $j$-th argument.
Note that, due to  (\ref{hyp}), it holds
\begin{eqnarray}\label{eq:h1}
	\begin{array}{ll}
		\inf\left\{a_j(x,t)\,:\, (x,t)\in\Pi, \,1\le j\le m\right\}\ge a_0,\\ [1mm]
		\sup\left\{a_j(x,t)\,:\, (x,t)\in\Pi,\,m+1\le j\le n\right\}\le -a_0,\\ [1mm]
		\inf\left\{|a_j(x,t)-a_k(x,t)|\,:\,  (x,t)\in\Pi,\,1\le j\ne k\le n\right\}\ge a_0.
	\end{array}
\end{eqnarray}                                                                                                                                                                                            
For  an arbitrary fixed initial time $s\in\bbbr$,
we supplement the system       (\ref{1lin})--(\ref{2lin})  with the initial conditions                                        
\begin{eqnarray}
	u(x,s) = \varphi(x), \quad x\in[0,1].
	\label{in}
\end{eqnarray}

We start with formulating a result from \cite{KL1}
about the existence and uniqueness of 	$L^2$-generalized  solutions
to the problem
(\ref{1lin}), (\ref{2lin}), (\ref{in}).
Due to \cite[Theorem 2.1]{ijdsde}, for any $\varphi\in C_0^1([0,1];\bbbr^n)$ and sufficiently smooth $a$ and $b$,
 the problem (\ref{1lin}), (\ref{2lin}), (\ref{in})  has a unique classical solution.

\begin{defn}\label{L2}
	Let $\varphi\in L^2((0,1);\bbbr^n)$ and $a,b\in BC^1(\Pi,\bbbm_n)$. A function $u$ that belongs to $C\left([s,\infty); L^2\left((0,1);\bbbr^n\right)\right)$ is called
	an {\rm $L^2$-generalized  solution} to the problem (\ref{1lin}), (\ref{2lin}), (\ref{in}) 
	if, for any sequence 
	$\varphi^l\in C_0^1([0,1];\bbbr^n)$ with 
	$\varphi^l\to\varphi$ in $L^2\left((0,1);\bbbr^n\right)$,
	the sequence $u^l\in C^1\left([0,1]\times[s,\infty);\bbbr^n\right)$
	of  classical solutions to
	(\ref{1lin}), (\ref{2lin}), (\ref{in})  with
	$\varphi$ replaced by $\varphi^l$   fulfills the convergence
	$$
	\|u(\cdot,\theta)-u^l(\cdot,\theta)\|_{L^2\left((0,1);\bbbr^n\right)} \to 0 \mbox{ as } l\to\infty,
	$$
	uniformly in $\theta$ varying in the range $s\le\theta\le s+s_1$, for every $s_1>0$.
\end{defn}

\begin{thm}\label{evol_g}\cite{KL1}
	Suppose that the coefficients  $a$ and $b$ 	of the system (\ref{1lin}) belong to $BC^1(\Pi;\bbbm_n)$ and 
	the conditions (\ref{eq:h1}) are fulfilled. 
	Then, for given $s\in\bbbr$ and $\varphi\in L^2\left((0,1);\bbbr^n\right)$, there exists a unique
	$L^2$-generalized  solution $u:\Pi\to\bbbr^n$ to the problem
	(\ref{1lin}), (\ref{2lin}), (\ref{in}).
\end{thm}

To formulate our main result, some notations are in order.
For given $j\le n$, $x \in [0,1]$, and $t \in \bbbr$, the $j$-th characteristic of (\ref{1lin})
passing through the point $(x,t)\in\Pi$ is defined
as the solution 
$$
\xi\in [0,1] \mapsto \omega_j(\xi)=\omega_j(\xi,x,t)\in \bbbr
$$ 
of the initial value problem
\begin{equation}\label{char}
\partial_\xi\omega_j(\xi, x,t)=\frac{1}{a_j(\xi,\omega_j(\xi,x,t))},\;\;
\omega_j(x,x,t)=t.
\end{equation}

For $i=0,1,2$ and $j\le n$, we introduce the notation
\begin{equation} \label{cd}
\hspace{-2mm}	c_j^i(\xi,x,t)=\exp \int_x^\xi
	\left[\frac{b_{jj}}{a_{j}} - i\frac{\partial_t a_{j}}{a_{j}^2} \right](\eta,\omega_j(\eta))\dd\eta,\ \ 
	d_j^i(\xi,x,t)=\frac{c_j^i(\xi,x,t)}{a_j(\xi,\omega_j(\xi))}.
\end{equation}
We write $c_j(\xi,x,t)$ and $d_j(\xi,x,t)$ for $c_j^0(\xi,x,t)$ and $d_j^0(\xi,x,t)$, respectively. 

Thanks to  (\ref{eq:h1}), the characteristic curve $\tau=\omega_j(\xi,x,t)$ passing through the point $(x,t)\in\Pi$ reaches the
boundary of $\Pi$ in two points with distinct ordinates. 
Let $x_j$
denote the abscissa of that point whose ordinate is smaller.
Notice that  the value of  $x_j$
does not depend on $x$ and $t$ and, due to (\ref{eq:h1}), it holds
\begin{equation}\label{*k}
x_j=\left\{
\begin{array}{rl}
0\quad &\mbox{ if }\ 1\le j\le m\\
1\quad &\mbox{ if }\ m<j\le n.
\end{array}
\right.
\end{equation}
Introduce operators $G_0, G_1, G_2 \in {\cal L}(BC(\bbbr, \bbbr^n))$ by
\begin{equation} \label{Ci}
 [G_i\psi]_j(t) = c_j^i(x_j, 1 - x_j, t)
\sum\limits_{k=1}^nr_{jk}\psi_k(\omega_j(x_j, 1 - x_j, t))\quad 
\mbox{for all } j \le n, \  \ i = 0,1,2.
	\end{equation}
\begin{defn}\label{stab}
Let $a, b\in BC^1(\Pi;\bbbm_n)$.
 	
	${\bf (i)}$
	We say that the problem (\ref{1lin})--(\ref{2lin}) is {\rm exponentially stable in~$L^2$} if there exist constants $M\ge 1$ and $\alpha>0$ such that, for all 
	$s\in\bbbr$ and $\varphi\in L^2((0,1);\bbbr^n)$, the $L^2$-generalized solution $u$ to the problem
 (\ref{1lin}), (\ref{2lin}), (\ref{in}) fulfills the exponential decay estimate
 \begin{equation}\label{stabily}
 \|u(\cdot,t)\|_{L^2((0,1);\bbbr^n)}\le Me^{-\alpha (t-s)}\|\varphi\|_{L^2((0,1);\bbbr^n)}\quad\mbox{for all } t\ge s.
 \end{equation}
 
${\bf (ii)}$ 	We say that the problem (\ref{1lin})--(\ref{2lin}) is {\rm robust exponentially stable in $L^2$} if there exist constants $\gamma>0$, $M\ge 1$, and $\alpha>0$ such that, for all 
$s\in\bbbr$, $\varphi\in L^2((0,1);\bbbr^n)$, and 
$\tilde a, \tilde b\in BC^1(\Pi;\bbbm_n)$ with 
$\|\tilde a - a\|_{BC^1(\Pi;\bbbm^n)} \le \gamma$ and $ \|\tilde b - b\|_{BC^1(\Pi;\bbbm^n)} \le \gamma$,
 the $L^2$-generalized solution $\tilde u$ to the problem
(\ref{1lin}), (\ref{2lin}), (\ref{in}) with $a$ and $b$ replaced by 
$\tilde a$ and $\tilde b$, respectively, 
fulfills the (uniform) exponential decay estimate
\begin{equation}\label{unif_stabily}
\|\tilde u(\cdot,t)\|_{L^2((0,1);\bbbr^n)}\le Me^{-\alpha (t-s)}\|\varphi\|_{L^2((0,1);\bbbr^n)}\quad\mbox{for all } t\ge s.
\end{equation}
\end{defn}

We are prepared to formulate the main result of the paper.

\begin{thm}\label{main} Let the conditions $\bf(A1)$--$\bf(A3)$ be fulfilled.
Assume also that the conditions
\begin{eqnarray} \label{G_i}
\|G_i\|_{{\cal L}(BC(\bbbr, \bbbr^n))} < 1
\end{eqnarray}
are satisfied  for $i=0, 1, 2$ and  
the  
problem (\ref{1lin})--(\ref{2lin})
is robust exponentially stable in $L^2$.
	Then there exist $\varepsilon>0$ and $\delta>0$ such that, 
	if $\|F_j(\cdot,\cdot,0)\|_{BC_t^2(\Pi)} \le\varepsilon$,
	 $\|H_j(\cdot,0)\|_{BC^2(\bbbr)}  \le\varepsilon$,
	 and $\|\partial_2 H_j(\cdot,0)\|_{BC^1(\bbbr)}  \le\varepsilon$
	  for all $j\le n$, then the problem~(\ref{1}), (\ref{2}) has a unique  
	  T-periodic classical
	  solution $u^*\in C^2_{per}(\Pi,\bbbr^n)$   such that $\|u^*\|_{BC^2(\Pi;\bbbr^n)}\le~\delta$.
\end{thm}

The paper is organized as follows. 
We complete the current introduction section with Sections  \ref{remarks} and \ref{related}  including related work and some remarks on our problem.
In Section \ref{linear} we give a thorough analysis of linear problems. Specifically, in Section \ref{periodic}  we prove that, if 
the homogeneous problem (\ref{1lin})--(\ref{2lin}) is  exponentially stable and dissipative, then  the nonhomogeneous linear problem  (\ref{1lin0}), (\ref{2lin}) has a unique smooth time-periodic solution, for every right hand side (Theorem \ref{lin-smooth}). 
In Section \ref{Lyapunovstability} we use the Lyapunov approach and derive sufficient conditions for the robust 
exponential stability 
of  (\ref{1lin})--(\ref{2lin}) in $L^2$ (Theorem \ref{L1}). 
We also give an example showing that the conditions 
of Theorem~\ref{lin-smooth} are consistent  (Example~\ref{ex1}). 
The time-periodic  solution survives under
small perturbations of the coefficients of the linear hyperbolic system
and satisfies a uniform a priori estimate, what is the scope of Theorem \ref{prop3}  in Section \ref{perturbation}.  Our main result stated in Theorem~\ref{main} about the local existence and the local uniqueness of small time-periodic solutions to the quasilinear problem is proved in Section \ref{quasilinear}.

\subsection{Remarks} \label{remarks}

\subsubsection{Condition (\ref{strict}) on the slope of the characteristic curves and correctness of the problem \reff{1}, (\ref{2}).} 
Accordingly to  (\ref{strict}) and (\ref{2}),
the number of the boundary conditions posed on the left hand side and on the right hand side of the boundary of $\Pi$ is strongly related
to the number of characteristic curves coming to the boundary as time decreases.
Remark that for the well-posedness of {\it time-periodic} hyperbolic problems this 
relation is not essential and can easily be dropped. We nevertheless keep
this relation in our assumptions, because our approach 
to solving  quasilinear problems 
involves  the corresponding   linear(ized) {\it initial boundary} value problems (in particular, the exponential stability issue  for 
those problems), where this relation is crucial  for the well-posedness.

\subsubsection{Dissipativity conditions (\ref{G_i}) and small divisors.}
The complication with small divisors naturally
appears when  addressing the existence  of
time-periodic solutions to hyperbolic PDEs.
The dissipativity or, the same, the non-resonant conditions (\ref{G_i}) on the coefficients of the linearized 
problem
are sufficient conditions  served to prevent  small divisors (or resonances) from coming up.  An interesting feature 
observed in our previous work (see 
\cite[Remark 1.4]{KR} and \cite[Section 3.6]{jee}) is that the number of those conditions in 
the nonautonomous setting essentially depends on the regularity of the
continuous solutions. Any higher order of the solution regularity requires  additional non-resonance conditions.

\subsubsection{Exponential stability  condition and existence of time-periodic solutions.}
The exponential stability condition not only ensures a "regular" behavior of the linearized problem, 
but also plays a crucial role in proving the solvability of the nonhomogeneous version of this problem. 
In particular, it contributes to the uniqueness part of the proof and, in combination with the 
Fredholm property,  leads to the desired solvability.

\subsubsection{Robust exponential stability condition and perturbation theorem.
}
The robustness part of the  exponential stability assumption in our main Theorem \ref{main} plays an important role in the proof of the perturbation 
Theorem \ref{prop3}. 
Whether this part can be omitted here seems to be an interesting open question.
Remark that the  conditions provided by Theorem \ref{prop3}
and derived by the Lyapunov approach
are sufficient conditions for the robust exponential stability required in Theorem \ref{main}. In other words,
Theorem \ref{prop3} justifies the robust exponential stability condition.

 \subsection{Related work} \label{related}
As mentioned above, one of the crucial assumptions in our main Theorem~\ref{main} is the existence 
and robustness of the exponential stability
 of the linearized problem.  A number of (sufficient) conditions for these assumptions are available in the literature. 
 These include the spectral criterion \cite{BC2,Lichtner,Neves,Renardy}, the resolvent criterion \cite[Theorem 1.11, p. 302]{Engel}, and
  the  Lyapunov approach
 (see, e.g. \cite{BC2,BCbook,BCN,DGL,GLTW,Ha,HSh1,HSh}). How these criteria and approaches can be used in concrete problems is illustrated by Example
 \ref{ex1} of the present paper, as well as by \cite[Examples 1.7 and 1.8]{KT}.
 
The robustness issue of exponential dichotomy, especially exponential stability, has been extensively studied in the literature (see, e.g., \cite{ChL2,H,PS}).
These results are hardly applicable to hyperbolic PDEs in the case of unbounded perturbations, mostly due to the well-known phenomenon of loss of regularity.
In \cite{KRT,KRT2} we proved a robustness theorem for the system (\ref{1lin})--(\ref{2lin}) with the boundary conditions (\ref{2lin}) of a smoothing type described  in \cite{Km,KL1}.
In \cite{KKR} 
the robustness issue is addressed in the case of  the
periodic boundary conditions, which are, evidently, not of the smoothing type. There, the authors give a number of sufficient conditions
for  robust exponential stability.
For 1D hyperbolic problems with more general boundary conditions,  the robustness issue
 remains a challenging open problem.
 
Note that the existence of small time-periodic solutions to quasilinear hyperbolic problems has been investigated by other approaches in
  \cite{Li,Qu0,Qu1,TY2,Y,Zhang}.
 The main difference lies in the linear analysis performed in Section \ref{linear}.
More precisely, the main point is  establishing conditions for the
 unique classical solvability of the nonhomogeneous time-periodic linearized problems.  A number of sufficient
 conditions  (different from ours) for this solvability 
 are obtained in \cite{BCN,DDTK1,KM,KL1,KR,KT,MV,PW}.
 
Speaking about potential applications of the models described by our system (\ref{1}), (\ref{2}), we would like to mention 
the one-dimensional version of the classical Saint-Venant system for shallow water in open channels and its generalizations \cite{BMPV,HPCAB},
the Saint-Venant--Exner model \cite{DDTK1,PW,SNC},
and the one-dimensional Euler equations \cite{GDL2,TY2,Y}.
The  system (\ref{1}), (\ref{2}) is also used to describe rate-type materials in viscoelasticity \cite{Cri,MV}  and the interactions between heterogeneous cancer cells \cite{BE}.

\section{Linear systems}\label{linear}

\subsection{Existence  of time-periodic solutions}\label{periodic}
\setcounter{claim}{0}
In this section we investigate a nonhomogeneous version of the problem 
(\ref{1lin})--(\ref{2lin}). Specifically, we consider
the general linear nonhomogeneous hyperbolic system 
\begin{equation}\label{1lin0}
\partial_tu  + a(x,t)\partial_x u + b(x,t) u = f(x,t), \quad x\in(0,1),
\end{equation}
subjected to the  linear nonhomogeneous boundary conditions
\begin{equation}\label{2p}
\begin{array}{ll}
u_{j}(0,t)= R_ju(\cdot,t) + h_j(t), \quad 1\le j\le m,\\ [1mm]
u_{j}(1,t)= R_ju(\cdot,t) + h_j(t), \quad m< j\le n.
\end{array}
\end{equation}

  The next theorem is a natural generalization of \cite[Theorem 1.6]{KT} for the case of  the inhomogeneous boundary conditions.

\begin{thm} \label{lin-smooth}
	${\bf (i)}$ 
	Assume that $a, b\in C^1_{per}(\Pi;\bbbm_n)$, $h\in C^1_{per}(\bbbr;\bbbr^n)$, and
	$ f\in C_{per}(\Pi;\bbbr^n)\cap BC^1_t(\Pi);\bbbr^n)$. Moreover, suppose 
	that the problem (\ref{1lin})--(\ref{2lin})  is exponentially stable in~$L^2$.
		If   the  conditions (\ref{G_i})
	are fulfilled  for $i=0$ and $i=1$, then
	the system  (\ref{1lin0})--(\ref{2p}) has a unique  classical time-periodic solution $u^*\in C_{per}^1(\Pi;\bbbr^n)$.
	Moreover, the a priori estimate 
	\begin{equation} \label{est211}
		\|u^*\|_{BC^1(\Pi;\bbbr^n)} \le L_1\left(\| f \|_{BC_t^1(\Pi;\bbbr^n)}+\| h \|_{BC^1(\bbbr;\bbbr^n)} \right)
	\end{equation}
	is fulfilled for a constant $L_1$ not depending on $f$ and $h$.
	
	${\bf (ii)}$ Assume, additionally, that $a,b\in BC_t^2(\Pi;\bbbm_n)$, 
	$h\in C^2_{per}(\bbbr;\bbbr^n)$, and
	$ f\in BC^1(\Pi;\bbbr^n)\cap BC^2_t(\Pi;\bbbr^n)$ and that the  condition
	(\ref{G_i}) is true for $i=2$. Then
	$u^* \in C^2_{per}(\Pi; \bbbr^n)$.
	Moreover, the a priori estimate 
	\begin{equation} \label{L6}
		\|u^*\|_{BC^2(\Pi;\bbbr^n)}   \le L_2
		\left(\| f\|_{BC^1(\Pi;\bbbr^n)}+\|\partial_t^2 f\|_{BC\left(\Pi;\bbbr^n\right)}+	\| h \|_{BC^2(\bbbr;\bbbr^n)} \right)	
	\end{equation}
	is fulfilled for a constant $L_2$ not depending on $f$ and $h$.
\end{thm}

\begin{proof}
	 The proof  
	  essentially repeats the  proof of 
	\cite[Theorem 1.6]{KT}, where a difference lies in the
	boundary conditions, which are inhomogeneous now.
We anyway provide this proof  for the sake of completeness and for our further purposes.

	First
write down the  problem (\ref{1lin0})--(\ref{2p})  in an equivalent operator form. To this end, introduce   operators    $C, D, S \in {\cal L}(BC(\Pi;\bbbr^n))$ and $P\in {\cal L}(BC(\bbbr;\bbbr^n))$ by
\begin{equation}\label{CDF}
\begin{array}{lll}
& &	[Cu]_j(x,t)= c_j(x_j,x,t)R_j u(\cdot,\omega_j(x_j,x,t))),
\\[2mm]
& &	[Ph]_j(x,t)= c_j(x_j,x,t)h_j(\omega_j(x_j,x,t))),
\\[2mm]
& & [Du]_j(x,t)=\displaystyle
-\int_{x_j}^{x}  d_j(\xi,x,t)\sum_{k\neq j}  b_{jk}(\xi, \omega_j(\xi,x,t))u_k(\xi, \omega_j(\xi,x,t)) \dd\xi,\\ [2mm]
& & [Sf]_j(x,t)=\displaystyle\int_{x_j}^{x}d_j(\xi,x,t)f_j(\xi, \omega_j(\xi,x,t)) \dd\xi.
\end{array}
\end{equation}
Then the problem (\ref{1lin0})--(\ref{2p}) after the integration  along the characteristic curves
can be written in the following operator form:
\begin{eqnarray} \label{oper}
u=Cu+Du+Ph+Sf.
\end{eqnarray}  
A function $u\in C_{per}(\Pi;\bbbr^n)$ satisfying the operator equation 
(\ref{oper}) pointwise is called a {\it continuous solution } to the  problem (\ref{1lin0})--(\ref{2p}).
 
 The proof of the theorem is divided into a number of claims. Specifically, the
proof of  Part ${\bf (i)}$ goes through Claims \ref{111}--\ref{5}, while the proof of  Part ${\bf (ii)}$  is done in Claim \ref{6}--\ref{7}.

In Claims \ref{111}--\ref{5} we, therefore, assume that  the conditions of  Part ${\bf (i)}$ of the theorem are fulfilled.

\begin{claim} \label{111}
	The operators $I-C: C_{per}(\Pi;\bbbr^n)\to C_{per}(\Pi;\bbbr^n)$  and 
	$I-C: C_{per}(\Pi;\bbbr^n)\cap BC^1_t(\Pi;\bbbr^n)\to C_{per}(\Pi;\bbbr^n)\cap BC^1_t(\Pi;\bbbr^n)$ defined by (\ref{CDF}) are bijective.
\end{claim}

\begin{subproof}
First note that 
$C$ maps both  $C_{per}(\Pi;\bbbr^n)$ into itself and $C_{per}(\Pi;\bbbr^n)\cap BC^1_t(\Pi;\bbbr^n)$ into itself.
Hence, to prove that $I-C$ is bijective from $C_{per}(\Pi;\bbbr^n)$  to itself,
it sufffices to show
that the system 
\begin{equation}\label{simpl}
	\begin{array}{rcl}
		u_j(x,t)&=&[Cu]_j(x,t)+g_j(x,t)\\ [2mm]
		&=&c_j(x_j,x,t)R_j u(\cdot,\omega_j(x_j,x,t))+g_j(x,t), \quad j\le n,
	\end{array}
\end{equation}
is uniquely solvable in $BC (\Pi;\bbbr^n)$ for any $g\in C_{per}(\Pi;\bbbr^n)$.
Since $Cu$ is a certain vector-function of $u_1,\dots,u_n$ evaluated at the boundary of $\Pi$, we first consider
(\ref{simpl})
at $x=0$ for $m<j\le n$ and at $x=1$ for $1\le j\le m$. 
On the account of  (\ref{Ci}),
 we get the following system:
\begin{equation}\label{simpl1}
	\begin{array}{rcl}
		z_j(t)&=& \displaystyle c_j(x_j,1 - x_j,t)\sum\limits_{k=1}^nr_{jk}z_k(\omega_j(x_j, 1 - x_j, t)) + g_j(1-x_j,t)\\ [4mm]
		&=& [G_0z]_j(t) + g_j(1-x_j,t), 
		\quad  j\le n,
	\end{array}
\end{equation}
with respect to the new unknown 
$$z(t) = (z_1(t),\dots,z_n(t))= (u_1(1,t),\dots, u_m(1,t), u_{m+1}(0,t),\dots,u_n(0,t)).$$
Evidently, the system (\ref{simpl}) is uniquely solvable in $BC (\Pi;\bbbr^n)$ if and only if
\begin{equation}\label{contr1}
	I-G_0 \mbox{ is bijective from }   BC^l(\bbbr;\bbbr^n) \mbox{ to }   BC^l(\bbbr;\bbbr^n)
\end{equation}
for $l=0$. This statement straightforwardly follows from the condition  (\ref{G_i}) for $i=0$, and this is the scope of the first part of the claim.

As it follows from the above, to prove  that the operator $I-C$ is bijective from 
$C_{per}(\Pi;\bbbr^n)\cap BC^1_t(\Pi;\bbbr^n)$ to itself, we have to prove the condition
(\ref{contr1}) for $l=1$.
To this end,  we use one of the equivalent norms in the space  $BC^1(\bbbr;\bbbr^n)$, namely the norm
\begin{equation}\label{beta}
	\|v\|_{\sigma} = \|v\|_{BC(\bbbr;\bbbr^n)} + \sigma \|\partial_t v\|_{BC(\bbbr;\bbbr^n)},
\end{equation}
where a positive constant $\sigma$ will be specified later. 
We are done if we  prove that  there exist  constants $\sigma<1$ and $\nu<1$ such that
\begin{equation}\label{*}
\|G_0 v\|_{BC(\bbbr;\bbbr^n)}+\sigma\left\|\frac{d}{dt} [G_0  v]\right\|_{BC(\bbbr;\bbbr^n)}
\le \nu\left(\|v\|_{BC(\bbbr;\bbbr^n)} + \sigma\|v^\prime\|_{BC(\bbbr;\bbbr^n)}\right)
\end{equation}
for all $ v \in  BC^1(\bbbr;\bbbr^n)$. For given $v\in BC^1(\bbbr;\bbbr^n)$
and $j\le n$, it holds
\begin{equation}\label{dtG0R}
	 \frac{d}{dt} [G_0 v]_j(t) =  [W v]_j(t)    +  [G_1v^\prime]_j(t),
\end{equation}
where the operator $G_1$ is given by  (\ref{Ci}) and the operator $W \in \mathcal L(BC(\bbbr; \bbbr^n))$ is defined by
$$[W v]_j(t) =  \partial_tc_j(x_j,1 - x_j,t) \sum\limits_{k=1}^nr_{jk}v_k(\omega_j(x_j, 1 - x_j, t)),
\quad j \le n. $$
Because of the assumption (\ref{G_i}),  we have  the bounds $\left\|G_0\right\|_{{\cal L}(BC(\bbbr; \bbbr^n))}<1$ and $\left\|G_1\right\|_{{\cal L}(BC(\bbbr; \bbbr^n))}<1$. Hence, we can fix $\sigma<1$ such that  $\left\|G_0\right\|_{{\cal L}(BC(\bbbr; \bbbr^n))}+
\sigma\left\|W\right\|_{{\cal L}(BC(\bbbr; \bbbr^n))}<1$. Then the number
$$
\nu=\max\left\{\left\|G_0\right\|_{{\cal L}(BC(\bbbr; \bbbr^n))}+
\sigma\left\|W\right\|_{{\cal L}(BC(\bbbr; \bbbr^n))}, \left\|G_1\right\|_{{\cal L}(BC(\bbbr; \bbbr^n))}
\right\}<1
$$
makes the estimate (\ref{*}) correct.
Indeed,
$$
\begin{array}{rcl}
	\|G_0v\|_{\sigma}
	&\le&\displaystyle
	\left\| G_0v\right\|_{BC(\bbbr;\bbbr^n)}
	+\sigma \| W v\|_{BC(\bbbr;\bbbr^n)}
	+ \sigma \left\|G_1 v^\prime\right\|_{BC(\bbbr;\bbbr^n)}\\ [2mm]
	&\le& \displaystyle\nu\left(\|v\|_{BC(\bbbr;\bbbr^n)} + \sigma\left\|v^\prime\right\|_{BC(\bbbr;\bbbr^n)}\right)=  \nu\| v\|_{\sigma},
\end{array}
$$
as desired.
It follows that
\begin{equation} \label{ots441}
	\begin{array}{rcl}
		\|(I - G_0)^{-1}v\|_{BC_t^1(\bbbr;\bbbr^n)} &\le& \displaystyle\frac{1}{\sigma}\|(I - G_0)^{-1}v\|_{\sigma} \le \frac{1}{\sigma(1 - \nu)}\|v\|_{\sigma}\\
		&\le& \displaystyle \frac{1}{\sigma(1 - \nu)}\|v\|_{BC_t^1(\bbbr;\bbbr^n)}.
	\end{array}
\end{equation}
This implies that the system (\ref{simpl1})
can be written in the form
\begin{eqnarray} \label{z1}
	z=(I-G_0)^{-1}\tilde g,
\end{eqnarray}
where $\tilde g(t)=\left(g_1(1,t),\dots,g_m(1,t),g_{m+1}(0,t),\dots,g_n(0,t)\right)$.
Substituting (\ref{z1}) into (\ref{simpl}), we obtain
\begin{equation} \label{uj}
	\begin{array}{rcl}
		u_j(x,t)&=&\displaystyle c_j(x_j,x,t)\sum\limits_{k=1}^nr_{jk}\left[(I-G_0)^{-1}\tilde g\right]_k(\omega_j(x_j,x,t))+g_j(x,t)\\ [5mm]
		&=&\left[(I-C)^{-1}g\right]_j(x,t)
	\end{array}
\end{equation}
for all $ j \le n$. 
Combining  (\ref{uj}) with (\ref{ots441}) gives the estimate
\begin{equation}\label{I-C--1}
	\|(I-C)^{-1}\|_{{\cal L}(BC^1_t(\Pi;\bbbr^n))}\le 1+\frac{1}{\sigma(1 - \nu)}\|C\|_{{\cal L}(BC^1_t(\Pi;\bbbr^n))}
\end{equation}
that we will use later on.
The proof of the second part of the claim is therewith complete. 
\end{subproof}

\begin{claim} \label{cl2}The operator $I-C-D$
	is Fredholm operator of index zero from $C_{per}(\Pi;\bbbr^n)$ to itself.
\end{claim}

\begin{subproof}
By Claim \ref{111},  the operator $I-C$ is bijective from  $C_{per}(\Pi;\bbbr^n)$ to itself. 
Hence, 
  $I- C -D: C_{per}(\Pi;\bbbr^n)\to C_{per}(\Pi;\bbbr^n)$ is Fredholm  opeeator of index zero 
if and only if $I-(I- C)^{-1} D: C_{per}(\Pi;\bbbr^n)\to C_{per}(\Pi;\bbbr^n)$ is Fredholm operator of index zero.
Accordingly  to the Fredholmness criterion stated in \cite[Proposition 1, pp. 298--299]{Zeidler}, the last statement is true  if there exist linear continuous operators $P_l, P_r: C_{per}(\Pi;\bbbr^n)\to C_{per}(\Pi;\bbbr^n)$  and compact operators
$Q_l,Q_r: C_{per}(\Pi;\bbbr^n)\to C_{per}(\Pi;\bbbr^n)$ such that
$$
P_l\left[I-(I-C)^{-1}D\right]=I+Q_l\quad\mbox{and}\quad \left[I-(I-C)^{-1} D\right]P_r=I+Q_r.
$$ 
Put
$
P_l=P_r=I+(I-C)^{-1}D.
$
Then we have
$$
\begin{array}{rcl}
P_l\left[I-(I-C)^{-1}D\right]=\left[I-(I-C)^{-1}D\right]P_r=I-(I-C)^{-1}D(I-C)^{-1}D.
\end{array}
$$
Accordingly to \cite[Proposition 1, pp. 298--299]{Zeidler}, the proof will, therefore, be done if we prove the compactness of the operator $(I-C)^{-1}D(I- C)^{-1} D$ or, on the account of the 
boundedness of  $(I-C)^{-1}$, the compactness of the operator $ D(I-C)^{-1} D: C_{per}(\Pi;\bbbr^n)\to C_{per}(\Pi;\bbbr^n)$. With this aim, we rewrite this operator
as follows:
$$
D(I-C)^{-1} D= D\left[I+C(I- C)^{-1}\right]D=D^2+ D C(I- C)^{-1}D.
$$
Because of the boundedness of  $(I-C)^{-1} D$, 
we are done if we prove  the compactness of 
$D^2$ and $ D C$. This follows from  \cite[Equation (4.2)]{KR}, where
it is shown that the operators
$D^2$ and $DC$ map continuously $C_{per}(\Pi;\bbbr^n)$ into $C_{per}^1(\Pi;\bbbr^n)$. The main idea of the Fredholmness proof above lies in the fact that we regularize the operator $I-(I-C)^{-1}D$  by means of the operators $P_l=P_r=I+(I-C)^{-1}D$.
The main point of this regularisation is that the Fredholmness of the regularized and original operators are equivalent. 
Specifically, the operator $D$, being a partial integral operator,  is not compact. At the same time,  the ``regularized'' operators $D^2$ and $DC$ are already integral operators and, hence,  are compact.
This completes the proof of the claim.
\end{subproof}	

\begin{claim} \label{3} The operator $I-C-D: C_{per}(\Pi;\bbbr^n)\to C_{per}(\Pi;\bbbr^n)$ 
	is injective.
\end{claim}

\begin{subproof}
Assume, contrary to this claim, 
that, for given $f$ and $h$, the equation (\ref{oper}) has two continuous periodic solutions, say $u^1$ and $u^2$. Then the difference
$u^1-u^2$ satisfies the equation (\ref{oper}) with $f=0$ and $h=0$ and, consequently, the following iterated equation:
\begin{eqnarray} \label{oper0}
u^1-u^2=C(u^1-u^2)+[DC + D^2](u^1-u^2).
\end{eqnarray}

As mentioned above, the operators
$D^2$ and $DC$ map continuously $C_{per}(\Pi;\bbbr^n)$ into $C_{per}^1(\Pi;\bbbr^n)$. Moreover, accordingly to Claim \ref{111},
the operator $(I-C)^{-1}$ maps $C_{per}(\Pi;\bbbr^n)\cap BC_t^1(\Pi;\bbbr^n)$ into itself. Hence, the equation (\ref{oper0}) implies that
$u^1-u^2\in BC_{t}^1(\Pi;\bbbr^n)$.  Using additionally that $u_1-u_2$ is a distributional solution to (\ref{1lin0}),
we get that $u^1-u^2\in C_{per}^1(\Pi;\bbbr^n)$ and, therefore,   that 
$u^1-u^2$ is a classical time-periodic solution to (\ref{1lin0})--(\ref{2p}). 

Let $s\in\bbbr$ be arbitrary fixed. Then the function $u^1(\cdot,s)-u^2(\cdot,s)$ belongs to $C([0,1];\bbbr^n)$ and, hence, to $L^2((0,1);\bbbr^n)$. By Definition \ref{L2}, the function $u^1-u^2$ is also the $L^2$-generalized solution to the problem (\ref{1lin0}), (\ref{2p}), (\ref{in}). Due to the exponential decay
estimate (\ref{stabily}),  
$u^1-u^2$ decays to zero as $t\to\infty$. As $u^1-u^2$ is $T$-periodic, this yields that
$u^1(x,t)= u^2(x,t)= 0$ for all $x\in[0,1]$ and $t\in \bbbr$. The injectivity is therewith proved.
\end{subproof}
		
\begin{claim} \label{4} The operator $I-C-D: C_{per}(\Pi;\bbbr^n)\to C_{per}(\Pi;\bbbr^n)$ 
	is bijective and, hence, the problem (\ref{1lin0})--(\ref{2p}) has a unique
	continuous time-periodic solution, say~$u^*$.
	 Moreover, $u^*$  satisfies the following a priori
	estimate: 
	\begin{equation} \label{BC}
	\|u^*\|_{BC(\Pi;\bbbr^n)} \le K\left(\| f \|_{BC(\Pi;\bbbr^n)} + 
	\| h \|_{BC(\bbbr;\bbbr^n)} \right)
	\end{equation}	
	with a constant $K$ not depending on $f$ and $h$.	
\end{claim}

\begin{subproof}
By Claim \ref{cl2}, the problem (\ref{oper})
satisfies the Fredholm Alternative in $C_{per}(\Pi;\bbbr^n)$.
Combining this with Claim~\ref{3} gives the desired bijectivity of the operator $I-C-D: C_{per}(\Pi;\bbbr^n)\to C_{per}(\Pi;\bbbr^n)$. In other words, the problem (\ref{1lin0})--(\ref{2p}) has a unique
continuous time-periodic solution $u^*$ given by the formula
$$
u^*=[I-C-D]^{-1}(Ph+Sf).
$$
It follows also that the map $(f,h)\in C_{per}(\Pi;\bbbr^n)\times C_{per}(\bbbr;\bbbr^n) \to 
u\in C_{per}(\Pi;\bbbr^n)$ is continuous. The a priori estimate 
(\ref{BC}) now easily follows.

\end{subproof}

\begin{claim} \label{5} The  continuous time-periodic solution $u^*$ to the problem (\ref{1lin0})--(\ref{2p}) belongs to $C^1(\Pi;\bbbr^n)$
	and satisfies the a priori estimate (\ref{est211}).
\end{claim}

\begin{subproof}
We iterate the equation (\ref{oper}) by substituting (\ref{oper}) into the second summand of this equation, and 
afterwards apply  the operator $(I-C)^{-1}$ to both sides of the obtained equation.  The resulting equation  for $u=u^*$ then reads 
\begin{eqnarray} \label{oper2}
u^*=(I-C)^{-1} \left[
DC+D^2
\right]u^*+(I-C)^{-1}\left[I+D
\right]\left[
Ph+Sf
\right].
\end{eqnarray}  
We again use  the smoothing property of the operators $DC$ and $D^2$, that map $BC$ into  $BC^1_t$ and satisfy the estimate
\begin{eqnarray} \label{ots3}
\left\|\partial_t\left[(DC + D^2)u\right]\right\|_{BC(\Pi;\bbbr^n)} \le K_{1} \|u\|_{BC(\Pi;\bbbr^n)}
\end{eqnarray}
for a constant $K_{1} $ which is uniform for all  $u\in BC(\Pi;\bbbr^n)$
(in particular, for $u=u^*$). Since $u^*$ is a distributinal
solution to (\ref{1lin0}),
this entails that $u^*$
is $BC^1$-smooth and  is, therefore, the classical solution to the problem 
(\ref{1lin0})--(\ref{2p}). The a priori estimate (\ref{est211}) then easily follows 
from the operator equation (\ref{oper2}), from the  estimates 
(\ref{I-C--1}), (\ref{BC}), (\ref{ots3}), and from the differential equation~(\ref{1lin0}).  
This finishes the proof of the claim.
\end{subproof}
The proof of Part ${\bf (i)}$ of the theorem is, therefore,
 complete. 
Now we assume that the conditions of  Part ${\bf (ii)}$ of the theorem are fulfilled.
	
\begin{claim} \label{6}
	The operator
	$I-C: C_{per}(\Pi;\bbbr^n)\cap BC^2_t(\Pi;\bbbr^n)\to C_{per}(\Pi;\bbbr^n)\cap BC^2_t(\Pi;\bbbr^n)$ is bijective.
\end{claim}

\begin{subproof}
The proof  follows the same line as the proof of Claim~\ref{111}, but now we have to show that the system (\ref{simpl}) is uniquely solvable in $BC^2_t (\Pi;\bbbr^n)$ or, the same, that the
system (\ref{simpl1}) is uniquely solvable in $BC^2 (\bbbr;\bbbr^n)$. To this end,
we  additionally use 
the condition (\ref{contr1})  for $l=2$.
Now we use  one of the equivalent norms in the space  $BC_t^2(\bbbr;\bbbr^n)$, namely the norm
\begin{equation}\label{beta2}
	\|v\|_{\sigma_1,\sigma_2} = \|v\|_{BC(\bbbr;\bbbr^n)} + \sigma_1 \|\partial_t v\|_{BC(\bbbr;\bbbr^n)}+ \sigma_2 \|\partial_t^2 v\|_{BC(\bbbr;\bbbr^n)},
\end{equation}
where the constants $\sigma_1<1$ and
$\sigma_2<1$ have to be chosen to satisfy the inequality
$$
\begin{array}{rr}
	\displaystyle
	\|G_0 v\|_{BC(\bbbr;\bbbr^n)}+\sigma_1\left\|\frac{d}{dt} [G_0v]\right\|_{BC(\bbbr;\bbbr^n)}+\sigma_2\left\|\frac{d^2}{dt^2} [G_0  v]\right\|_{BC(\bbbr;\bbbr^n)}
	\\ [5mm]
	\le \nu\left(\|v\|_{BC(\bbbr;\bbbr^n)} + \sigma_1\|v^\prime\|_{BC(\bbbr;\bbbr^n)}+ \sigma_2\|v^{\prime\prime}\|_{BC(\bbbr;\bbbr^n)}\right)
\end{array}
$$
for a constant  $\nu<1$ and all $ v \in  BC^2(\bbbr;\bbbr^n)$. To show the existence of such constants, we differentiate the function 
(\ref{dtG0R})  and get
$$
\begin{array}{rcl}
\displaystyle	\frac{d^2}{dt^2} [G_0 v]_j(t) &=& \displaystyle \partial_t^2 c_j(x_j,1 - x_j,t)\sum\limits_{k=1}^nr_{jk}v_k(\omega_j(x_j, 1 - x_j, t))   +  [G_2v^{\prime\prime}]_j(t)\\ [4mm]
&&+ \left(\partial_t c_j(x_j,1 - x_j,t)\partial_t\omega_j(x_j;1 - x_j,t)+\partial_t c_j^1(x_j,1 - x_j,t)\right)\\ [3mm]
&&\times
\displaystyle\sum\limits_{k=1}^nr_{jk}v_k^\prime(\omega_j(x_j, 1 - x_j, t)) .
\end{array}
$$
Further we proceed similarly to the proof of  Claim~\ref{111}, but now
taking additionally into account the condition (\ref{G_i}) for $i=2$.
Moreover, similarly to (\ref{I-C--1}),  we
 derive the estimate
\begin{equation}\label{BC2t}
	\|(I-C)^{-1}\|_{{\cal L}(BC^2_t(\Pi;\bbbr^n))}\le 1+\frac{1}{\max\{\sigma_1,\sigma_2\}(1 - \nu)}\|C\|_{{\cal L}(BC^2_t(\Pi;\bbbr^n))}.
\end{equation}
\end{subproof}
		
\begin{claim} \label{7} The classical  time-periodic solution $u^*$ to the problem (\ref{1lin0})--(\ref{2p}) belongs to $C^2_{per}(\Pi;\bbbr^n)$
	and satisfies the a priori estimate (\ref{L6}).
\end{claim}

\begin{subproof}
One can easily show that, under the regularity assumptions on the coefficients of  (\ref{1lin0}), the operators $DC$
and $D^2$ have the smoothing property of the type that they map $BC^1_{t}(\Pi;\bbbr^n)$ into $BC^2_{t}(\Pi;\bbbr^n)$.
Moreover, 
the following smoothing estimate is fulfilled:
\begin{eqnarray} \label{ots4} 
\left\|\partial_t^2\left[(DC + D^2)u\right]\right\|_{BC(\Pi;\bbbr^n)} \le K_{2} \|u\|_{BC^1_t(\Pi;\bbbr^n)}
\end{eqnarray}
for some $K_2>0$ and all $u\in BC^1_t(\Pi;\bbbr^n)$ (in particular, for $u=u^*$).
Since $u^*$ satisfies 
(\ref{oper2}), the claim now follows from Claim \ref{6}, the 
estimates 
 (\ref{est211}), (\ref{BC2t}), (\ref{ots4}), and the differential system~(\ref{1lin0})
 and its differentiations in $t$ and $x$.
\end{subproof}

The proof of the theorem is complete.	
		
\end{proof}

	\subsection{Lyapunov function and sufficient conditions for the robust exponential stability in $L^2$}\label{Lyapunovstability}
	
	In this section we consider general linear initial-boundary value problem of the type
	(\ref{1lin}), (\ref{2lin}), (\ref{in}). 
		Write $u = (u^1, u^2),$ where
	$u^1 = (u_1,\dots,u_m)$ and $u^2 = (u_{m+1},\dots,u_n).$
	We will also use the following notation:
	\begin{eqnarray} \label{linop3}
		z(t) = 
		\left[\begin{array}{c}
			u^1(1,t)  \\
			u^2(0,t) \\
		\end{array}\right], \ 
		J_0 = \left[
		\begin{array}{cc}
			R_{11}  & R_{12} \\
			0 &  I \\
		\end{array}
		\right], \ 
		J_1 = \left[
		\begin{array}{cc}
			I &  0 \\
			R_{21}  & R_{22} \\ 
		\end{array}
		\right],
	\end{eqnarray} 
	where
	$
	R_{11}=(r_{jk})_{j,k=1}^m, \ R_{12}=(r_{jk})_{j=1,k=m+1}^{m,n},\ R_{21}=(r_{jk})_{j=m+1,k=1}^{n,m},\ R_{22}=(r_{jk})_{j,k=m+1}^n,
	$
	and $r_{jk}$ are the  boundary coefficients from (\ref{eq:R}).
	Then the boundary conditions (\ref{2lin}) can be written in the form
	\begin{eqnarray} \label{linop2}
		u(0,t) = J_0 z(t), \quad u(1,t) = J_1 z(t).
	\end{eqnarray}
	
	\begin{thm} \label{L1}
		Let $a, b\in BC^1(\Pi;\bbbm_n)$.
		Assume that there exist  positive constants $\beta_i$,  $i\le 4$,  and a  diagonal matrix $V(x,t)=\diag(V_1(x,t),\dots,V_n(x,t))$  with continuously differentiable
		entries in $\Pi$
		such that the following conditions are satisfied: \\
		\vspace{1mm}
		${\bf (i)}$
		$\beta_1\|v\|^2\le  \langle Vv,v\rangle \le \beta_2\|v\|^2 $
		for all $(x,t) \in \Pi$ and $v \in \bbbr^n$,\\
		\vspace{1mm}
	${\bf (ii)}$ $\left\langle \left(\partial_t V  +  \partial_x(Va) - V b - b^TV \right) v,v\right\rangle < -\beta_3 \|v\|^2$
	for all $(x,t) \in \Pi$ and $v\in\bbbr^n$,\\
		\vspace{1mm}
		${\bf (iii)}$ 
		$\left\langle \left( J_0^T V(0,t)a(0,t) J_0  -  J_1^T V(1,t)a(1,t) J_1\right)v,v\right\rangle  <-\beta_4\|v\|^2$  for all $t  \in \bbbr$ and $v \in \bbbr^n$. \\
		\vspace{1mm}
		Then the problem  (\ref{1lin})--(\ref{2lin}) is robust exponentially stable in  $L^2((0,1),\bbbr^n)$.
	\end{thm}  
	\begin{proof} 
	Due to	the continuous embedding	$C_0^1([0,1];\bbbr^n)\hookrightarrow L^2\left((0,1);\bbbr^n\right)$
	and  Definition \ref{L2}
		of the $L^2$-generalized solution, it suffices to prove the estimate  (\ref{unif_stabily})  for any 
		$\varphi\in C_0^1([0,1];\bbbr^n)$. Fix an arbitrary $\varphi\in C_0^1([0,1];\bbbr^n)$ and denote by $u$
		the classical solution to the problem  (\ref{1lin}), (\ref{2lin}), (\ref{in}) (see \cite[Theorem 2.1]{ijdsde}).

		Set ${\cal V}(t) = \int_0^1 \langle V(x,t)u(x,t),u(x,t)\rangle  \dd x$.
		The derivative of ${\cal V}(t)$ along the trajectories of the system (\ref{1lin})--(\ref{2lin}) is calculated directly, as follows:
		$$
		\begin{array}{rcl}
			\displaystyle	 \frac{d {\cal V}}{d t}& =& \displaystyle\int_0^1 \left[\left\langle \partial_t V u,u\right\rangle + \left\langle V\partial_t u,u\right\rangle+
			\langle Vu, \partial_t u\rangle\right]\dd x = \\ [3mm]
			& =& \displaystyle \int_0^1 \left[\left\langle\partial_t V u,u\right\rangle- \left\langle V a\partial_x u,u\right\rangle- \left\langle V b u,u\right\rangle-
			\left\langle V u, a\partial_x u\right\rangle- \left\langle V u, b u\right\rangle\right]\dd x.
		\end{array}
		$$
		Since 
		\begin{eqnarray*}
			& & \displaystyle\int_0^1 \partial_x \langle Vau,u\rangle \dd x = \langle Vau,u\rangle\Bigl|_0^1 \\  
			& & = \int_0^1 \langle\partial_x(Va)u,u\rangle \dd x +
			\int_0^1 \langle Va  \partial_x u,u\rangle d x +  \int_0^1\langle Vau, \partial_x u\rangle \dd x,
		\end{eqnarray*}
		the following equality holds:
		\begin{eqnarray*}
			\langle Vau,u\rangle\Bigl|_0^1 - \int_0^1 \langle \partial_x(Va)u,u\rangle \dd x =
			\int_0^1 \langle Va  \partial_x u,u\rangle \dd x +  \int_0^1\langle Vu, a\partial_x u\rangle \dd x.
		\end{eqnarray*} 
		It follows that
		$$
		\begin{array}{rcl}
			\displaystyle		\frac{d {\cal V}}{d t} &=&\displaystyle -  \langle Vau,u\rangle\Bigl|_0^1 + \int_0^1 \langle
			\left[\partial_t V +  \partial_x(Va)-V b -  b^TV \right] u,u\rangle\dd x.
		\end{array}  
		$$
		Taking into account (\ref{linop2}), the boundary term in the last equation is computed as follows: 
		$$
		\begin{array}{rcl}
			- \langle Vau,u\rangle\Bigl|_0^1
			&=& \left\langle V(0,t)a(0,t)u(0,t), u(0,t)\right\rangle -  
			\langle  V(1,t)a(1,t)u(1,t), u(1,t)\rangle  \\ [3mm]
			&	 = &\langle V(0,t)a(0,t) J_0 z(t), J_0 z(t)\rangle -  \langle V(1,t)a(1,t)J_1z(t), J_1z(t)\rangle \\ [3mm]
			&	=&
			\left\langle\left[ J_0^T  V(0,t)a(0,t)J_0  -  J_1^T V(1,t)a(1,t) J_1\right]z(t),z(t)\right\rangle.
		\end{array}
		$$
		By Assumption ${\bf (iii)}$, $- \langle Vau,u\rangle\Bigl|_0^1\le 0$. Then, thanks to the 
		assumptions ${\bf (i)}$ and~${\bf (ii)}$, it holds 
		$$
		 \frac{d {\cal V}(t)}{d t} \le - \beta_3  \|u(\cdot,t)\|_{L^2((0,1),\bbbr^n)}^2 \le - \frac{\beta_3}{\beta_2} {\cal V}(t)\quad\mbox{for all } t\in\bbbr.
		 $$
		Consequently,  ${\cal V}(t) \le {\cal V}(s) e^{- \frac{\beta_3}{\beta_2}(t - s)}$ for all $ t \ge s$.
		Now, using  the assumption ${\bf (i)}$, we derive  the  inequality
		\begin{equation}\label{en1}
			\|u(t)\|_{L^2((0,1),\bbbr^n)}^2 \le \frac{\beta_2}{\beta_1} 
			e^{- \frac{\beta_3}{\beta_2}(t - s)} \|u(s)\|_{L^2(0,1),\bbbr^n)}^2 \quad\mbox{for all } t \ge s.
		\end{equation} 
		Since  the conditions ${\bf (ii)}$ and ${\bf (iii)}$  are stable with respect to  small perturbations of 
		$a$ and $b$, the estimate (\ref{en1}) is fulfilled for sufficiently small perturbations of 
		$a$ and $b$. This comletes the proof.
				\hide{
		On the account of (\ref{eps}), the following simple inequalities are true: 
		$$
		\|v(t)\|_{L^2((0,1),\bbbr^n)}^2\le \|u(t)\|_{L^2((0,1),\bbbr^n)}^2+\varepsilon\quad\mbox{for all } s\le t \le t_1,
		$$
		and
		$$
		\|u(s)\|_{L^2((0,1),\bbbr^n)}^2\le \|\varphi\|_{L^2((0,1),\bbbr^n)}^2+\varepsilon.
		$$
		Combining the last two inequalities with (\ref{en1}) gives 
		$$
		\|v(t)\|_{L^2((0,1),\bbbr^n)}^2\le \frac{\beta_2}{\beta_1} 
		e^{- \frac{\beta_3}{\beta_1}(t - s)} \|\varphi\|_{L^2(0,1),\bbbr^n)}^2 +2\varepsilon\quad\mbox{for all } s\le t \le t_1.
		$$
		As $\varepsilon$  and $t_1>s$ are arbitrary fixed, the proof is complete.
	}
	\end{proof}

	\begin{ex}\label{ex1}
		Here we show that the  conditions  of Theorem \ref{L1} are consistent. More precisely, our aim
		is to show that the robust exponential stability assumption does not contradict to the dissipativity conditions (\ref{G_i}).
				
		Consider the $2\times 2$-hyperbolic system
		\begin{eqnarray*}
			& & \partial_t u_1 + (2-x)\,\partial_x u_1 + 2\sin t \,u_2 = 0\\
			& & \partial_t u_2 -(2+\sin t)\,\partial_x u_2 - \sin t\, u_1 + 2 u_2 = 0
		\end{eqnarray*}
		with the boundary conditions
		$$u_1(0,t) = \frac{1}{e^3}u_1(1,t) + \frac{1}{2e^3}u_2(0,t), \quad
		u_2(1,t) = \frac{1}{e^3}u_1(1,t) + \frac{1}{e^3}u_2(0,t).$$
		
		Let us first show that the  Lyapunov function $V(x,t)$ as in Theorem \ref{L1} can be chosen to be the unit diagonal matrix 
		$V(x,t) = \diag(1,1)$. In other words,  we intend to show that the matrix $V$ satisfies  all conditions of Theorem \ref{L1}.
		
		The condition ${\bf (i)}$ is evident. To check ${\bf (ii)}$,  we calculate the matrix 
		\begin{eqnarray*}
		\partial_t V  +  \partial_x(Va) - V b - b^TV &=&	\left[
			\begin{array}{cc}
				-1  & 0 \\
				0 &  0\\
			\end{array}
			\right] -
			\left[
			\begin{array}{cc}
				0 & 2\sin t \\
				-\sin t &  2 \\
			\end{array}
			\right]  -
			\left[
			\begin{array}{cc}
				0  & -\sin t \\
				2\sin t &  2 \\
			\end{array}
			\right]  \\ &=&			\left[
			\begin{array}{cc}
				-1  & -\sin t \\
				-\sin t &  -4 \\
			\end{array}
			\right].
		\end{eqnarray*} 
	Hence, the matrix $\left(\partial_t V  +  \partial_x(Va) - V b - b^TV \right) $ is uniformly negative definite, which means that the condition ${\bf (ii)}$ 	of Theorem \ref{L1} is satisfied.
	
	To verify the condition ${\bf (iii)}$, we calculate the matrix
		\begin{eqnarray*}
&J_0^T V(0,t)a(0,t) J_0  -  J_1^T V(1,t)a(1,t) J_1 =	\left[
		\begin{array}{cc}
			r_{11}  & 0 \\
			r_{12} &  1\\
		\end{array}
		\right] \times
		\left[
		\begin{array}{cc}
			2 & 0 \\
			0 &\  -2-\sin t \\
		\end{array}
		\right]  \times
		\left[
		\begin{array}{cc}
			r_{11}	 & 	r_{12} \\
			0 &  1 \\
		\end{array}
		\right] & \\ &-			\left[
		\begin{array}{cc}
			1  & 	r_{21} \\
			0& 	r_{22} \\
		\end{array}
		\right]     
	\times
	\left[
	\begin{array}{cc}
		1 & 0 \\
		0 &\  -2-\sin t \\
	\end{array}
	\right]  \times
	\left[
	\begin{array}{cc}
		1	 & 0 \\
		r_{21}  & r_{22}  \\
	\end{array}
	\right] & \\ & =
		\left[
	\begin{array}{cc}
	2	r_{11} ^2+(2+\sin t)r_{21} ^2-1&\ 2	r_{11}r_{12}+(2+\sin t)r_{21}r_{22}\\
		2	r_{11}r_{12}+(2+\sin t)r_{21}r_{22}&\ 2	r_{12} ^2+(2+\sin t)r_{22} ^2-2
	\end{array}
	\right] .
	&
\end{eqnarray*}
Since $r_{11}=r_{21}=r_{22}=1/e^3$ and $r_{12}=1/2e^3$, 
	the matrix $J_0^T V(0,t)a(0,t) J_0  -  J_1^T V(1,t)a(1,t) J_1$ is strictly uniformly negative definite,
	which is the scope of the condition ${\bf (iii)}$. 
	 
	It remains to check that the dissipativity conditions (\ref{G_i})  are fulfilled for $i=0,1,2$. Indeed, for any $i=0,1,2$ and $\psi\in BC(\bbbr;\bbbr^n)$ with $\|\psi\|_{BC(\bbbr;\bbbr^n)}=1$,
	it holds
	$$
	\begin{array}{rcl}
		\left|[G_i\psi]_2(t)\right|&\le&
		\displaystyle\exp\left\{\int_{0}^{1}\left[\frac{2}{2+\sin(\omega_j(\eta))}+i\frac{\cos(\omega_j(\eta)) }{(2+\sin(\omega_j(\eta))^2}\right]\dd \eta\right\}\\ [5mm]
		&&\times\left\|R\right\|_{\mathcal{L}(BC([0,1];\bbbr^n);\bbbr^n)}
		\le e^2\left\|R\right\|_{\mathcal{L}(BC([0,1];\bbbr^n);\bbbr^n)}<1,
			\end{array}
	$$
	as desired. Analogously one checks the dissipativity condition (\ref{G_i}) for  $[G_i\psi]_1(t)$.

	\end{ex}

	 \subsection{Perturbation theorem}\label{perturbation}
In this section we prove that the time-periodic solution to the problem (\ref{1lin0})--(\ref{2p}) survives under sufficiently small perturbations
of the coefficients $a$ and $b$ and that the inverse operator is bounded uniformly in those perturbations. With this aim, we  consider the
perturbed system
\begin{equation}\label{1p}
	\partial_tu + \tilde a(x,t)\partial_x u+ \tilde b(x,t) u= f(x,t), \quad x\in(0,1),
\end{equation}
where
$\tilde a=\diag(\tilde a_1,\dots,\tilde a_n)$ and
$\tilde b=(\tilde b_{jk})_{j,k=1}^n$ are matrices of $T$-periodic real-valued functions.

\begin{thm} \label{prop3}
	Assume that
there exists $\gamma>0$ such that the problem 
(\ref{1lin})--(\ref{2lin}) is robust exponentially stable in $L^2$
 in the sense of Definition \ref{stab} $({\bf ii})$.
Moreover, assume that the condition (\ref{eq:h1}) is fulfilled
with $\tilde a$    in place of $a$  for all $\tilde a$ with $\|\tilde a - a\|_{BC^2(\Pi;\bbbm_n)}\le \gamma$.
	Then the following is true.
	
		$(\bf i)$
Let the conditions (\ref{G_i})
be fulfilled  for $i=0, 1$. Then there exists $\gamma_1 \le\gamma$ such that, for given 	
$ f\in C_{per}(\Pi;\bbbr^n)\cap BC^1_t(\Pi;\bbbr^n)$,  $h\in C^1_{per}(\bbbr;\bbbr^n)$, and	$\tilde a, \tilde b \in C_{per}^2(\Pi,\bbbm_n)$ with $\|\tilde a - a\|_{BC^2(\Pi;\bbbm_n)} \le \gamma_1$ and $ \|\tilde b - b\|_{BC^2(\Pi;\bbbm_n)} \le \gamma_1, $
 the problem (\ref{1p}), (\ref{2p}) has a
unique classical $T$-periodic  solution, say $\tilde u$. Moreover,
$\tilde u$  satisfies the uniform a priori estimate
	\begin{equation} \label{est21p}
	\|\tilde u\|_{BC^1(\Pi;\bbbr^n)} \le L_3\left(\| f \|_{BC_t^1(\Pi;\bbbr^n)} + 
	\| h \|_{BC^1(\bbbr;\bbbr^n)} \right)
	\end{equation}
	with a constant $L_3$ not depending on $\tilde a, \tilde b$, $f$, and $h$.

	$(\bf ii)$
	Let the conditions (\ref{G_i})
	be fulfilled  for $i=0,1,2$. Then there exists $\gamma_1 \le\gamma$ such that, for given 	
	 $f \in C_{per}^1(\Pi;\bbbr^n)\cap BC^2_t(\Pi;\bbbr^n)$, $h \in C_{per}^2(\bbbr,\bbbr^n)$, and
	$\tilde a, \tilde b \in C_{per}^2(\Pi;\bbbm_n)$ with $\|\tilde a - a\|_{BC^2(\Pi;\bbbm_n)} \le \gamma_1$ and $ \|\tilde b - b\|_{BC^2(\Pi;\bbbm_n)} \le \gamma_1 $, the classical solution $\tilde u$	
	to the problem 	(\ref{1p}), (\ref{2p}) belongs to $C_{per}^2(\Pi,\bbbr^n)$	(see Theorem \ref{lin-smooth} $({\bf ii})$) and satisfies the uniform a priori estimate 
	\begin{equation} \label{est22p}
	\|\tilde u\|_{BC^2(\Pi;\bbbr^n)} \le L_4\left(\| f \|_{BC^1(\Pi;\bbbr^n)} + \|\partial_t^2 f \|_{BC(\Pi;\bbbr^n)}
	+ \| h \|_{BC^2(\bbbr;\bbbr^n)} \right)
	\end{equation}
	with a constant $L_4$ not depending on $\tilde a, \tilde b$, $f$, and $h$.
	
\end{thm}

\begin{proof}
	Let the conditions of Part 	$(\bf i)$ of the theorem be fulfilled.  Denote by  $\widetilde \omega_j(\xi,x,t)$ the continuous solution
	to the problem (\ref{char}) with $a_j$ replaced by $\tilde a_j$.
	We will write $\widetilde G_i$, $\widetilde C$,  $\widetilde D$, $\widetilde P$, and $\widetilde S$ for the operators $G_i$, $C$, $D$, $P$, and $S$ defined by (\ref{Ci}) and (\ref{CDF})
	where $a$, $b$, and $\omega$ are replaced by $\tilde a$ and $\tilde b$, and $\widetilde\omega$, respectively.
	Fix 
	$\gamma_1\le\gamma$ so small  that the inequalities
$\|\widetilde G_0\|_{{\cal L}(BC(\bbbr; \bbbr^n))}<c_0$ and $\|\widetilde G_1\|_{{\cal L}(BC(\bbbr; \bbbr^n))}<c_0$ are fulfilled
for some $c_0<1$ and all $\tilde a, \tilde b \in C_{per}^2(\Pi,\bbbm_n)$ 
varying in the range $\|\tilde a - a\|_{BC^2(\Pi;\bbbm_n)} \le \gamma_1$ and $ \|\tilde b - b\|_{BC^2(\Pi;\bbbm_n)} \le \gamma_1$.
Let 	
$ f\in C_{per}(\Pi;\bbbr^n)\cap BC^1_t(\Pi;\bbbr^n)$,  $h\in C^1_{per}(\bbbr;\bbbr^n)$, and 	$\tilde a, \tilde b \in C_{per}^2(\Pi;\bbbm_n)$  in the range $\|\tilde a - a\|_{BC^2(\Pi;\bbbm_n)} \le \gamma_1$ and $ \|\tilde b - b\|_{BC^2(\Pi;\bbbm_n)} \le \gamma_1$ be arbitrary fixed.
Then, by Part $({\bf i})$ of Theorem \ref{lin-smooth},
the problem (\ref{1p}), (\ref{2p}) has a
unique classical $T$-periodic solution, that will be denoted by $\tilde u$. Moreover, $\tilde u$ satisfies the a priori estimate (\ref{est21p}) with a constant
$L_3$ not depending on  $f$ and~$h$ (but possibly depending on
$\tilde a$ and $\tilde b$). To show that $L_3$ can be chosen to be the same for all
$\tilde a$ and $\tilde b$, we first prove for $\tilde u$ the uniform upper bound as in (\ref{BC}).
In other words, we have to prove  that the following inequality is true for all $g\in C_{per}(\Pi,\bbbr^n)$:
	\begin{equation}\label{ots1p}
\left\|[I- \widetilde C- \widetilde D]^{-1}g\right\|_{BC(\Pi;\bbbr^n)} \le K_3\| g \|_{BC(\Pi;\bbbr^n)},
\end{equation}
where the constant $K_3$ can be chosen to be the same for all  
 $\tilde a$ and $\tilde b$ in the $\gamma_1$-ball of $a$ and $b$. 

To prove (\ref{ots1p}),  
assume, on the contrary,
	that there exist sequences  $(\tilde a^l)$ and $(\tilde b^l)$ with $C_{per}^2(\Pi;\bbbm_n)$-elements, 
	 $(g^l)$ with $C_{per}^2(\Pi;\bbbr^n)$-elements, and the corresponding them sequences $(\widetilde\omega^l(\xi))$
	 of functions defining characteristic curves of the system 
	 (\ref{1p}), such that
	\begin{equation}\label{ots2p}
	\|v^l\|_{BC(\Pi;\bbbr^n)} \ge l\| g^l \|_{BC(\Pi;\bbbr^n)}.
	\end{equation}
	Here $v^l=(I- \widetilde C^l- \widetilde D^l)^{-1}g^l$,
 $\|v^l\|_{BC^2(\Pi;\bbbr^n)} = 1$ for all $l\in\bbbn$, and $(\widetilde C^l)$ and $(\widetilde D^l)$   are the corresponding sequences of perturbed operators.
 It follows  that
	\begin{equation}\label{ots3p}
	\|[I- \widetilde C^l- \widetilde D^l]v^l\|_{BC(\Pi;\bbbr^n)} \le \frac{1}{l}\| v^l \|_{BC(\Pi;\bbbr^n)}.
	\end{equation}
	Since  the sequences  $(\tilde a^l)$ and $(\tilde b^l)$  are uniformly bounded in $C^2_{per}(\Pi;\bbbm_n)$,
	 by the Arzel\`a–Ascoli theorem, there exist
	 subsequences 
	of $(\tilde a^l)$ and  $(\tilde b^l)$
	(here and below we will keep the same notation for the subsequences) converging in $C_{per}^1(\Pi;\bbbm_n)$ to functions, say
	$a^*$ and $b^*$. Hence, the corresponding subsequence of ($\widetilde\omega^l(\xi)$) converges in $BC^1([0,1]\times\Pi;\bbbr^n)$ to, say
	$\omega^*(\xi)$. 
	Note that for all $l\ge 1$ it holds
	$$
\widetilde	\omega^l(\xi;x,t+T)=\widetilde	\omega^l(\xi;x,t)+T
	$$
	and, hence
	$$
	\omega^*(\xi;x,t+T)=	\omega^*(\xi;x,t)+T.
	$$
	This, in its turn, implies that 
	the corresponding subsequence of $(v^l)$ converges in $C^1_{per}(\Pi;\bbbr^n)$ to a function 
	$v^*\in C_{per}^1(\Pi;\bbbr^n)$. 	
	Note that $v^* \not\equiv 0$,
	since $\|v^l\|_{BC^2(\Pi;\bbbr^n)} = 1.$ 
	By $C^*$ and $D^*\in {\cal L}(C_{per}(\Pi;\bbbr^n))$ we denote the
		limit operators of $(\widetilde C^l)$ and $(\widetilde D^l)$, respectively. 
		Passing to  the limit as $l\to\infty$ results in the equation
	$$[I - C^* - D^*]v^* = 0.$$ 
	This means that  $v^*$ is a nontrivial continuous time-periodic solution to the homogeneous 
	problem (\ref{1p}), (\ref{2p}) (where $f\equiv 0$ and $h\equiv 0$) with $a^*$ and $b^*$ in place of $\tilde a$ and $\tilde b$, respectively. This, in its turn, contradicts to Claim \ref{3} in the proof Theorem \ref{lin-smooth} stating that the operator $I-C^*-D^*: C_{per}(\Pi;\bbbr^n)\to C_{per}(\Pi;\bbbr^n)$ 
	is injective.
		We, therefore, conclude that the a priori estimate (\ref{ots1p}) 
	is fulfilled with a constant
	$K_3$ not depending on  $f$, $h$, $\tilde a$, $\tilde b$. This means, in particular, that
	\begin{equation}\label{ots2p}
		\begin{array}{rcl} 
	\| u^*\|_{BC(\Pi;\bbbr^n)} & = & \left\| [I - \widetilde C - \widetilde  D]^{-1}( \widetilde Ph+\widetilde Sf)\right\|_{BC(\Pi;\bbbr^n)} \\ [3mm]
	& \le & K_{4}\left(\| f \|_{BC(\Pi,\bbbr^n)} + \| h \|_{BC(\bbbr,\bbbr^n)} \right).
	\end{array} 
\end{equation}
for a constant $K_4$ not depending on $f$, $h$, $\tilde a$, $\tilde b$. 
	
	Next we prove the uniform a priori estimate (\ref{est21p}).
The operators 
 $\widetilde C$ and $\widetilde D$ are bounded in the norm of  ${\cal L}(BC^1_t)$ uniformly in $\tilde a$ and $\tilde b$
 varying in the range 
 $\|\tilde a - a\|_{BC^2(\Pi;\bbbm_n)} \le \gamma_1$ and $ \|\tilde b - b\|_{BC^2(\Pi;\bbbm_n)} \le \gamma_1$. 
 On the account of  (\ref{I-C--1}) and (\ref{ots3}), the following estimates:
\begin{equation}\label{I-C--2}
\|(I-\widetilde C)^{-1}\|_{{\cal L}(BC^1_t(\Pi;\bbbr^n))}\le 1+\frac{1}{\sigma(1 - \nu)}\|\widetilde C\|_{{\cal L}(BC^1_t(\Pi;\bbbr^n))}
\end{equation}
and
\begin{eqnarray} \label{ots31}
\left\|\partial_t\left[(\widetilde D\widetilde C + \widetilde D^2)u\right]\right\|_{BC(\Pi;\bbbr^n)} \le K_{5} \|u\|_{BC(\Pi;\bbbr^n)}
\end{eqnarray}
are fulfilled for some $\sigma<1$, $\nu<1$, and $K_{5}>0$ uniformly in all 
$\tilde a$, $\tilde b$, and $u$ under consideration. 
In particular, (\ref{ots31}) is true for $u=u^*$.
Further, following the  proof of Claim \ref{4} in 
Section \ref{periodic} with 
$\tilde a$ and $\tilde b$ in place of $a$ and $b$
and taking into account the  estimates (\ref{I-C--2}) and (\ref{ots31}), we derive  the desired estimate~(\ref{est21p}). Part 	$(\bf i)$ of the theorem is therewith proved.

To prove Part  $(\bf ii)$, 
we  fix 
$\gamma_1\le\gamma$ so small  that there exists  $c_0<1$ such that
$\|\widetilde G_i\|_{{\cal L}(BC(\bbbr; \bbbr^n))}<c_0$   for 
$i=0,1,2$ and  all $\tilde a, \tilde b \in C_{per}^2(\Pi,\bbbm_n)$ 
varying in the range $\|\tilde a - a\|_{BC^2(\Pi;\bbbm_n)} \le \gamma_1$ and $ \|\tilde b - b\|_{BC^2(\Pi;\bbbm_n)} \le \gamma_1$. 
Then, by Claim~$(\bf ii)$ of Theorem \ref{lin-smooth}, the classical solution $\tilde u$	
to the problem 	(\ref{1p}), (\ref{2p}) belongs to $C_{per}^2(\Pi,\bbbr^n)$.
To prove the  estimate (\ref{est22p}), we take
additionally into account  that the estimate (\ref{ots4}) with $C$ and $D$ replaced by $\tilde C$ and $\tilde D$
is uniform in $\tilde a$ and $\tilde b$ varying in the $\gamma_1$-balls of $a$ and $b$, respectively,
with appopriately chosen $K_{2}$.
 Then,
repeating the proof of  Claim~\ref{7} in Section \ref{periodic} 
with 
$\tilde a$ and $\tilde b$ in place of $a$ and $b$, respectively, results in the desired estimate (\ref{est22p}).

The proof of Theorem \ref{prop3} is complete.
	\end{proof}
   
  \section{Quasilinear problems: proof of Theorem \ref{main}.}\label{quasilinear}
\setcounter{claim}{0}

We first  rewrite the system (\ref{1}) in a suitable form, namely
$$\partial_t u + A(x,t,u)\partial_x u + B(x,t, u)u = f(x,t),$$
where $f(x,t) = F(x,t, 0)$ and $B(x,t, u)u =  F(x,t, 0)- F(x,t, u)$. By our assumptions, the functions $f$ and $B$ are $T$-periodic in 
$t$.	

Let $a_0$ be a constant fulfilling the condition (\ref{hyp}).  Let  $\gamma_1$ and $L_4$ be constants
as in Theorem \ref{prop3} $(\bf ii)$.

Put $u^0(x,t) =0.$  Determine a sequence of iterations  $(u^{k}(x,t))_{k=1}^\infty$ 
such that any subsequent iteration $u^{k+1}$
is the unique
classical $T$-periodic solution   to the linear system
\begin{eqnarray} \label{qua1}
	\partial_t u^{k+1}+ A^k\partial_x u^{k+1}  + B^ku^{k+1} = f(x,t),\ \  k=0,1,2,\dots
\end{eqnarray}
subjected to the boundary conditions 
\begin{equation}\label{qua11}
	\begin{array}{ll}
		u^{k+1}_{j}(0,t)= R_ju^{k+1}(\cdot,t)  + H_j^k(t), \quad 1\le j\le m,\\ [1mm]
		u^{k+1}_{j}(1,t)= R_ju^{k+1}(\cdot,t)   + H_j^k(t), \quad m< j\le n,
	\end{array}
\end{equation}
where $H^k_j(t)=H_j(t, Q_j(t)u^k(\cdot,t))$. 
Here and below we will simply write $A^k$ and $B^k$ 
for $A(x,t,u^k(x,t))$ and  $B(x,t,u^k(x,t))$,  respectively.

Let us separate the linear part of   the boundary operator $H_j(t,Q_j(t)u)$ near $u=0$, by writing $H_j$ as follows:
\begin{equation}\label{approx}
	\begin{array}{rcl}
H_j(t,Q_j(t)u(\cdot,t))&=& h_j(t) + \partial_2H_j(t,0)Q_j(t) u(\cdot,t)\\ [2mm]
&&\hskip-8mm\displaystyle+ 	\left[Q_j(t)	u(\cdot,t)\right]^2 \int_0^1
\d_2^2H_j(t,\sigma Q_j(t)u(\cdot,t))
\dd\sigma,\quad j\le n,
	\end{array}
\end{equation}
where $h(t)=H(t,0)$.
Write
$$
\|Q\|=\sum_{k=0}^2\sup\left\{\left\|\frac{d^k}{dt^k}Q(t) \right\|_{{\cal L}\left(BC([0,1];\bbbr^n),\bbbr^n\right)}\, :\, t\in\bbbr\right\} 
$$
and set
\begin{equation}\label{HQ2}
	\begin{array}{ll}
\varrho_1=\displaystyle\max\left\{	\left\| \partial_2H_j(\cdot,0) \right\|_{BC^2(\bbbr;\bbbr^n)}\,:\,j\le n\right\}, \\ [3mm]
	\displaystyle	
\varrho_2=	\sup\left\{\left\|\partial_2^2H_j(\cdot,v(\cdot))\right\|_{BC^2(\bbbr;\bbbr^n)}\,: \,   
		 \|v\|_{BC^2(\bbbr)}\le \delta_0\|Q \|, \,  j\le n \right\},
	\end{array}
\end{equation}
where the constant $\delta_0$ is as in Assumption ${\bf (A1)}$.
We will suppose that the number $\varrho_1$ is small enough to fulfill  the inequality $\varrho_1\|Q\|< 1/L_4$ for some $L_4$  satisfying the a priori estimate (\ref{est22p}).
Set $a^\psi (x,t) = A(x,t,\psi (x,t)) - A(x,t,0)$ and $b^\psi(x,t) = B(x,t,\psi (x,t)) - B(x,t,0). $
Since the functions $A$ and $B$ are $C^2$-smooth, then
there exists a positive real 
\begin{equation}\label{delta1}
\delta_1 \le \min\left\{\delta_0, \frac{1}{ \varrho_2\|Q\|}\left( \frac{1}{L_4\|Q\|} - \varrho_1\right)\right\}
\end{equation}
such that, for all $\psi\in C^2_{per}( \Pi; \bbbr^n)$
with $\|\psi\|_{BC^2(\Pi;\bbbr^n)} \le \delta_1$,
we have
\begin{equation}\label{eps1}\| a^\psi \|_{BC^2( \Pi; \bbbm^n)}    \le\gamma_1, \quad 
\| b^\psi \|_{BC^2( \Pi; \bbbm^n)}   \le \gamma_1.
\end{equation}
Hence, by Theorem \ref{prop3} ${\bf (ii)}$, for given $\psi \in C^2_{per}(\Pi;\bbbr^n)$,  the system
$$\partial_t u + A(x,t,\psi)\partial_x u + B(x,t, \psi)u = f(x,t)$$
with the boundary conditions (\ref{2p})  has a unique classical solution 
$u^\psi$, which belongs to $C^2_{per}(\Pi,\bbbr^n)$  and
satisfies the a priori estimate (\ref{est22p}).

We divide the proof into a number of claims and proceed similarly to the proof of  \cite[Theorem 1]{jee}. 

\begin{claim}\label{8}
	If $\varrho_1\|Q\|< 1/L_4$ for some $L_4$  satisfying the a priori estimate (\ref{est22p}) and
	\begin{equation} \label{qua4}
		\begin{array}{c}
				\| f \|_{BC^1(\Pi;\bbbr^n)} + \|\partial_t^2 f \|_{BC(\Pi;\bbbr^n)}
		+ \| h\|_{BC^2(\bbbr;\bbbr^n)} \\ [2mm]
		\displaystyle\le \frac{\delta_1}{L_4}\left[1-\left(\varrho_1 +  \delta_1\varrho_2\|Q\|\right)\|Q\|L_4\right],
	\end{array}
	\end{equation}
	then 	the problem (\ref{qua1})--(\ref{qua11})  generates  a sequence $(u^k)$ of  classical time-periodic solutions $u^k\in C^2_{per}(\Pi;\bbbr^n)$
 such that
	\begin{equation}\label{qua5}
		\| u^k\|_{BC^2( \Pi; \bbbr^n)}\le \delta_1 \quad \mbox{for all } k\ge 1.
	\end{equation}
\end{claim}

\begin{subproof} 	
	\hide{
First write the expression for $H^k$ in a form that is desirable for our further purposes.
	The formula   (\ref{approx1}) reads
	\begin{equation}\label{approx2}
		\begin{array}{rcl}
			\displaystyle			\d_2p(t,v(t))=	\d_2H(t,v(t))-\d_2H(t,0)=
			v(t)\int_0^1\partial_2^2H(t,\sigma_1v(t))\dd\sigma _1.
		\end{array}
	\end{equation}
	As $p(t,0)=0$, we derive from (\ref{approx}), (\ref{approx1}), and (\ref{approx2})
	that
	\begin{equation}\label{approx3}
		\begin{array}{rcl}
			H^k(t)&=&h(t) + \partial_2H(t,0)Q(t) u^k(\cdot,t)+ p(t,Q (t)u^k(\cdot,t))-p(t,0)\\ [2mm]
			&=&\displaystyle h(t)  + \partial_2H(t,0)Q(t) u^k(\cdot,t)+
			Q(t)	u^k(\cdot,t)\int_0^1\partial_2 p\left(t,\sigma Q(t)u^k(\cdot,t)\right)\dd\sigma \\ [3mm]
			&=&\displaystyle h(t)  + \partial_2H(t,0)Q(t) u^k(\cdot,t)\\ [2mm]
			&&+\displaystyle
			Q(t)	u^k(\cdot,t) Q(t)u^k(\cdot,t) \int_0^1\sigma \int_0^1\d_2^2H(t,\sigma_1\sigma Q(t)u^k(\cdot,t))
			\dd\sigma\dd\sigma_1.
		\end{array}
	\end{equation}
}
	Note that the first iteration $u^1$  satisfies the system (\ref{1lin0})--(\ref{2p}), where $a(x,t)=A(x,t,0)$, $b(x,t)= \partial_3 F(x,t,0)$,
	and $h(t)=H(t,0)$.
	Due to Theorem \ref{lin-smooth} {\bf (ii)}, the last system has a unique time-periodic
	classical solution   $u^1 \in C^2_{per}(\Pi;\bbbr^n)$.
	Moreover,  this solution satisfies the a priori estimate~(\ref{L6}) and, obviously, the weaker  estimate (\ref{est22p}). Hence,
	if  $ f$  and  $ h$ satisfy the estimate~(\ref{qua4}), then, on the account of (\ref{delta1}), 
	$u^1$ satisfies the estimate~(\ref{qua5}) with $k=1$. Furthermore, due to (\ref{eps1}), it holds that
	$$
	\|A^1 - A^0\|_{BC^2( \Pi; \bbbm^n)} \le \gamma_1, \quad
	\|B^1 - B^0\|_{BC^2( \Pi; \bbbm^n)} \le \gamma_1.
	$$
	Now we use Theorem \ref{prop3} {\bf (ii)} to state that there exists a unique  
	classical  time-periodic solution $u^2\in C^2_{per}(\Pi,\bbbr^n)$ to the system (\ref{qua1})--(\ref{qua11}) with $k=1$.
	Moreover,  due to  (\ref{est22p}), (\ref{HQ2}),  (\ref{qua4}), and  (\ref{qua5}),  the following desired a priori estimate is fulfilled:
	\begin{eqnarray*} 
		\hspace{-1mm}\|u^2\|_{BC^2(\Pi;\bbbr^n)} &\le& L_4\Bigl(\| f \|_{BC^1(\Pi;\bbbr^n)} + \|\partial_t^2 f \|_{BC(\Pi;\bbbr^n)} 
		+ \|H^1\|_{BC^2(\bbbr;\bbbr^n)}\Bigl) \\
		&\le& \delta_1\left[1-\left(\varrho_1 +  \delta_1\varrho_2\|Q\|\right)\|Q\|L_4\right]+
	\delta_1	\left(\varrho_1 +  \delta_1\varrho_2\|Q\|\right)\|Q\|L_4= \delta_1.	
	\end{eqnarray*}
	We again use  (\ref{eps1}) to immediately conclude  that
	$$
	\|A^2 - A^0\|_{BC^2( \Pi; \bbbm^n)} \le \gamma_1, \quad
	\|B^2 - B^0\|_{BC^2( \Pi; \bbbm^n)} \le \gamma_1.
	$$
	
	Proceeding by induction, assume that the problem (\ref{qua1})--(\ref{qua11})
	has  a unique  classical time-periodic solution $u^k$  belonging to
	$C^2_{per}(\Pi,\bbbr^n)$ and satisfying the bound (\ref{qua5}). 
	This implies, in particular, that
	$$
	\|A^k - A^0\|_{BC^2( \Pi; \bbbm^n)} \le \gamma_1, \quad
	\|B^k - B^0\|_{BC^2( \Pi; \bbbm^n)} \le \gamma_1.
	$$ 
	By  Theorem \ref{prop3} {\bf (ii)},  the problem (\ref{qua1})--(\ref{qua11}) has  a unique   
	classical time-periodic
	solution $u^{k+1}\in C^2_{per}(\Pi,\bbbr^n)$. That this solution  fulfills the bound (\ref{qua5}),
	with $k+1$ in place of $k$, again follows  from (\ref{est22p}), (\ref{HQ2}),  (\ref{qua4}), and  (\ref{qua5}).
	\end{subproof}

\begin{claim}\label{9}
	There exists 
	$$\varepsilon_1\le \frac{\delta_1}{L_4}\left[1-\left(\varrho_1 +  \delta_1\varrho_2\|Q\|\right)\|Q\|L_4\right]
	$$ 
	such that, if
	$ \| f \|_{BC^1(\Pi;\bbbr^n)} + \|\partial_t^2 f \|_{BC(\Pi;\bbbr^n)} + \|h\|_{BC^2(\bbbr,\bbbr^n)}  < \varepsilon_1,
	$
	then  	the sequence  $(u^k)$ 
	converges in   $C_{per}^1(\Pi;\bbbr^n)$ to a classical $T$-periodic solution to the problem (\ref{1}), (\ref{2}).
\end{claim}

\begin{subproof}
Set 
$z^{k+1} = u^{k+1} - u^k.$
The function $z^{k+1}$ belongs to $C_{per}^2(\Pi;\bbbr^n)$ and satisfies the  system
\begin{eqnarray}
	& & \partial_ t z^{k+1} + A^k \partial_ x z^{k+1}+ B^kz^{k+1} = f^k(x,t)
	\label{qua21}
\end{eqnarray}
and the boundary conditions  
\begin{equation}\label{qua211}
	\begin{array}{ll}
		z^{k+1}_{j}(0,t)= R_jz^{k+1}(\cdot,t) +  H_j^k(t) - H_j^{k-1}(t), \quad
		1\le j\le m,\\ [1mm]
		z^{k+1}_{j}(1,t)= R_jz^{k+1}(\cdot,t) + H_j^k(t) - H_j^{k-1}(t), \quad m< j\le n,
	\end{array}
\end{equation}
for all $\ t\in \bbbr$, where 
\begin{eqnarray*}
	  f^k  & =& - \left(B^k -   B^{k-1}\right)u^k- \left(A^k - 
	A^{k-1}\right)\partial_ xu^k\\
	 & =& -
	 \left[\left\langle\int_0^1\,\nabla_uB_{ij}\left(x,t,\sigma u^k+(1-\sigma)u^{k-1}\right)\dd\sigma, z^k\right\rangle\right]_{i,j=1}^nu^k\\  [2mm]
	 & & -
	  \left[\left\langle\int_0^1\nabla_uA_{ij}\left(x,t,\sigma u^k+(1-\sigma)u^{k-1}\right)\dd\sigma, z^k\right\rangle
	 \right]_{i,j=1}^n\partial_xu^k\,
\end{eqnarray*}
and 
\begin{eqnarray*}
&&	 H^k_j(t)- H_j^{k-1}(t)=  
 \partial_2H_j(t,0)Q_j(t)z^k(\cdot,t) \\
 &&\qquad\quad+
 Q_j(t)	z_j^k(\cdot,t)\int_0^1\biggl\{
\left[\sigma Q_j(t)u^k(\cdot,t) +(1-\sigma)Q_j(t)u^{k-1}(\cdot,t)  \right]\\
&&\qquad\quad\times
 \int_0^1\partial_2^2H_j\left(t,\sigma_1\left[\sigma Q_j(t)u^k(\cdot,t) +(1-\sigma)Q_j(t)u^{k-1}(\cdot,t)\right] \right)\dd\sigma_1 \biggl\} \dd\sigma.
 \end{eqnarray*}

Our aim  is now to show that the problem (\ref{qua21})--(\ref{qua211}) for $z^{k+1}$ fulfills the conditions of the perturbaion Theorem~\ref{prop3}~$(\bf i)$.
For that, we first derive an upper bound for $u^k$ in terms of $f$ and~$h$.
Accordingly to the proof of Claim \ref{8}, the first iteration $u^1$ satisfies the inequality
$$\|u^1\|_{BC^2(\Pi;\bbbr^n)} \le L_4 \left( \| f \|_{BC^1(\Pi;\bbbr^n)} + \|\partial_t^2 f \|_{BC(\Pi;\bbbr^n)} +  \| h\|_{BC^2(\bbbr;\bbbr^n)}\right),$$
while the second iteration $u^2$ satisfies the following inequality (see (\ref{HQ2}) and (\ref{delta1})):
\begin{eqnarray*} 
	\|u^2\|_{BC^2(\Pi;\bbbr^n)} &\le& L_4( \| f \|_{BC^1(\Pi;\bbbr^n)} 
	+ \|\partial_t^2 f \|_{BC(\Pi;\bbbr^n)} + \|H^1\|_{BC^2(\bbbr;\bbbr^n)}) \\ [2mm]
	&\le& L_4 \left( 1 + (\varrho_1 +  \delta_1\varrho_2\|Q\|)\|Q\|L_4\right)
	\nonumber\\[1mm]
	& &\times\left( \| f \|_{BC^1(\Pi;\bbbr^n)} + \|\partial_t^2 f \|_{BC(\Pi;\bbbr^n)} + \| h\|_{BC^2(\bbbr;\bbbr^n)}\right).
\end{eqnarray*}
Proceeding by induction, assume that the problem (\ref{qua1})--(\ref{qua11})
has  a unique  classical time-periodic solution $u^{k-1}$  belonging to
$C^2_{per}(\Pi,\bbbr^n)$ and satisfying the bound 
\begin{equation}\label{uk-1}
\begin{array}{rcl} 
	\|u^{k-1}\|_{BC^2(\Pi;\bbbr^n)} &\le &L_4\Bigl( 1 + (\varrho_1 + \delta_1\varrho_2\|Q\|)\|Q\|L_4+\dots
	\\ [2mm]&&\hskip-5mm
	+\left[(\varrho_1 + \delta_1\varrho_2\|Q\|)\|Q\|L_4\right]^{k-2}  \Bigr) 
	\\ [2mm]
	&& \hskip-5mm \times\left(\| f \|_{BC^1(\Pi;\bbbr^n)} + \|\partial_t^2 f \|_{BC(\Pi;\bbbr^n)}  + \| h\|_{BC^2(\bbbr;\bbbr^n)}\right).
\end{array}
\end{equation}
and, hence, the bound (\ref{qua5}).
Using (\ref{HQ2}), (\ref{delta1}), (\ref{qua5}), and (\ref{uk-1}),  we derive the following  estimate for  $u^{k}$:
\begin{equation} \label{uk}
	\begin{array}{ll}
		\|u^{k}\|_{BC^2(\Pi;\bbbr^n)} \le 
		L_4\left( \| f \|_{BC^1(\Pi;\bbbr^n)} + \|\partial_t^2 f \|_{BC(\Pi;\bbbr^n)}  + 
		\|H^{k-1}\|_{BC^2(\bbbr;\bbbr^n)}\right)  \\[2mm]
		\qquad\le L_4 \left( 1 + (\varrho_1 + \delta_1\varrho_2\|Q\|)\|Q\|L_4+\dots+\left[(\varrho_1 + \delta_1\varrho_2\|Q\|)\|Q\|L_4\right]^{k-1} \right)  \\[3mm]
		\qquad	\quad\times  \left( \| f \|_{BC^1(\Pi;\bbbr^n)} + \|\partial_t^2 f \|_{BC(\Pi;\bbbr^n)}  + \| h\|_{BC^2(\bbbr;\bbbr^n)} \right) \\[2mm]
		\qquad\le\displaystyle \frac{L_4}{1 -(\varrho_1 + \delta_1\varrho_2\|Q\|)\|Q\|L_4}\\[4mm]
		\qquad	\quad\times
		\left( \| f \|_{BC^1(\Pi;\bbbr^n)} + \|\partial_t^2 f \|_{BC(\Pi;\bbbr^n)}  + \| h\|_{BC^2(\bbbr;\bbbr^n)} \right).
	\end{array}
\end{equation}
It follows that
\begin{equation}\label{qua22}
	\begin{array}{ll}
		\|f^{k}\|_{BC^1(\Pi;\bbbr^n)} \\[3mm] 
		\qquad\le N_1 \left(\|u^k\|_{BC^1(\Pi;\bbbr^n)} + \|\partial_{x} u^k\|_{BC^1(\Pi;\bbbr^n)}\right)\|z^k\|_{BC^1(\Pi;\bbbr^n)} \\[3mm]
		\qquad \le  \displaystyle \frac{N_1 L_4}{1 -(\varrho_1 +  \varrho_2\delta_1)\|Q\|L_4} \|z^k\|_{BC^1(\Pi;\bbbr^n)} \\[4mm]
		\qquad\quad \times \left(\| f \|_{BC^1(\Pi;\bbbr^n)} + \|\partial_t^2 f \|_{BC(\Pi;\bbbr^n)} 
		+ \| h\|_{BC^2(\bbbr;\bbbr^n)} \right), \\[4mm]
		\|H^k- H^{k-1}\|_{BC^1(\bbbr;\bbbr^n)} 
		\le\displaystyle   \|z^k\|_{BC^1_t(\Pi;\bbbr^n)}
\displaystyle			\biggl(\varrho_1 + \frac{N_2 L_4}{1 -(\varrho_1 +  \varrho_2\delta_1)\|Q\|L_4} \\[4mm]
\qquad \quad
		\times	\left( \| f \|_{BC^1(\Pi;\bbbr^n)} + \|\partial_t^2 f \|_{BC(\Pi;\bbbr^n)}  + \| h\|_{BC^2(\bbbr;\bbbr^n)} \right)	\biggr) ,
	\end{array}
\end{equation}
where the constants $N_1$ and $N_2$ depend 
on $A$, $B$,  $H$, and $\delta_1$ but not on $z^k$.

Applying now the  estimate (\ref{est21p}) to
$z^{k+1}$ and using (\ref{qua22}), we derive  the
inequality
\begin{eqnarray}\label{qua23}
	\begin{array}{ll}
		\|z^{k+1}\|_{BC^1(\Pi;\bbbr^n)} \le L_3\left(\| f^k \|_{BC_t^1(\Pi;\bbbr^n)} + 
		\|H^k -H^{k-1} \|_{BC^1(\bbbr;\bbbr^n)} \right)  \\[3mm]
		\qquad \le\displaystyle \|z^k\|_{BC^1(\Pi;\bbbr^n)}\biggl(\varrho_1+\frac{L_3L_4(N_1 + N_2)}{1 -(\varrho_1 +  \varrho_2\delta_1)\|Q\|L_4}  \\[4mm]
		\qquad\times\left( \| f \|_{BC^1(\Pi;\bbbr^n)} + \|\partial_t^2 f \|_{BC(\Pi;\bbbr^n)}  + \|  h\|_{BC^2(\bbbr;\bbbr^n)} \right)
		\biggr).
	\end{array}
\end{eqnarray}
Set
\begin{equation}\label{qua24}
	\varepsilon_1=\min\left\{\frac{\delta_1}{L_4}\left[1-\left(\varrho_1 +  \delta_1\varrho_2\|Q\|\right)\|Q\|L_4\right]
	, \frac{(1-\varrho_1)\left[1 -(\varrho_1 +  \varrho_2\delta_1)\|Q\|L_4\right]}{L_3L_4(N_1 + N_2)}\right\}.
\end{equation}
If
$  \| f \|_{BC^1(\Pi;\bbbr^n)} + \|\partial_t^2 f \|_{BC(\Pi;\bbbr^n)}  +  \|h\|_{BC^2} < \varepsilon_1,$
then,  due to (\ref{qua23}), the sequence $(z^k)$ is contracting and, hence,  tends to zero in 
$C^1_{per}(\Pi;\bbbr^n).$

Consequently,
the sequence $(u^k)$ converges to a function $u^*$ in $C^1_{per}(\Pi;\bbbr^n)$. It follows that
$u^*$ is a classical solution to the problem 
(\ref{1}), (\ref{2}).
The proof of the claim is complete.
\end{subproof}
		
\begin{claim}\label{10}
	There exist  positive reals  $\varepsilon \le \varepsilon_1$ and $\delta\le\delta_1 $ such that, if 
	$\varrho_1<\varepsilon$ and
	$$\| f \|_{BC^1(\Pi;\bbbr^n)} + \|\partial_t^2 f \|_{BC(\Pi;\bbbr^n)}  + \|h\|_{BC^2(\bbbr;\bbbr^n)}\le \varepsilon,$$
	then
	the solution $u^*$ belongs to $C^2_{per}(\Pi;\bbbr^n)$ and satisfies the estimate
	\begin{eqnarray}
		\| u^*\|_{BC^2(\Pi;\bbbr^n)} \le \delta.
		\label{*1**}
	\end{eqnarray}
\end{claim}
\begin{subproof}
	It is sufficient  to show that the sequence $(u^k)$ converges in $C^2_{per}(\Pi;\bbbr^n)$.
First show that the sequence 
$(\partial_tu^{k+1})$
converges in $BC^1_t(\Pi;\bbbr^n) \cap C_{per}(\Pi;\bbbr^n)$. With this aim, we
differentiate the system (\ref{qua1})--(\ref{qua11})
with respect to $t$ and  write down the resulting system with respect to $v^{k+1}=  \partial_tu^{k+1}$, namely
\begin{equation}\label{qua32}
	\partial_t v^{k+1} + A^k\partial_x v^{k+1} 
	+ B^{1}(x,t,u^k)v^{k+1}=g^{1k}(x,t,v^k)v^{k+1}+g^{2k}(x,t,v^k),
\end{equation}
\begin{equation}\label{qua33}
	\begin{array}{rcl}
		v_{j}^{k+1}(0,t)&=& R_jv^{k+1}(\cdot,t) + h^\prime_j(t) \\
		&&+ 
		\displaystyle\frac{d}{dt}\left[  \partial_2H_j(t,0)Q_j(t) u^k(\cdot,t)+ p_j^k(t)\right], \ 1\le j\le m,\\ [3mm]
		v_{j}^{k+1}(1,t)&=& R_jv^{k+1}(\cdot,t) + h^\prime_j(t)    \\
		&&\displaystyle+\frac{d}{dt}\left[ \partial_2H(t,0)Q(t) u^k(\cdot,t)+ p_j^k(t)\right], \ m< j\le n,
	\end{array}
\end{equation}
where
\begin{equation} \label{qua333}
	\begin{array}{rcl}
	B^{1}(x,t,u^k)&=&  B^k -  \partial_2A^k (A^k)^{-1}, \\ [2mm]
	g^{1k}(x,t,v^k) &=& \left(\partial_3 A^k v^k \right)(A^k)^{-1},
	\nonumber\\ [2mm]
	g^{2k}(x,t,v^k) &=& -\partial_2B^ku^{k+1}+\partial_2A^k\,(A^k)^{-1} B^ku^{k+1} - 
	\left(\partial_3 B^k  v^k\right)u^{k+1} \nonumber\\ [2mm]
	&&  +  \partial_t f - \partial_2A^k \,(A^k)^{-1}f  
	+ \left(\partial_3 A^k\, v^k\right) (A^k)^{-1}(B^k u^{k+1} -f), \\ [2mm]
	p_j^k(t)&=&\displaystyle \left[Q_j(t)	u^k(\cdot,t)\right]^2 \int_0^1
	\d_2^2H_j(t,\sigma Q_j(t)u^k(\cdot,t))
	\dd\sigma,\\ [3mm]
	\partial_3 A^k v^k&=&\diag\left(\langle\nabla_uA_{1}(x,t,u^k),v^k\rangle,\dots,\langle\nabla_uA_{n}(x,t,u^k),v^k\rangle\right),
	\\ [3mm]
	\partial_3 B^k  v^k&=&\left[\langle\nabla_uB_{ij}(x,t,u^k),v^k\rangle\right]_{i,j=1}^n.
	\end{array}
\end{equation} 
Similarly to $B^k$, we will write also $B^{1k}$ for $B^{1}(x,t,u^k)$.
It is evident that the sequence 
$(v^{k+1})$ of $T$-periodic solutions to the problem (\ref{qua32})--(\ref{qua33})
converges in $BC^1_t(\Pi;\bbbr^n)$ if and only if the sequence $(w^{k+1})$ with $w^{k+1}=\partial_tv^{k+1}$
converges in $C_{per}(\Pi;\bbbr^n)$.  We, therefore, will prove the convergence of 
$w^{k+1}$. To this end,  
by differentiating the system (\ref{qua32}) in a distributional sense in $t$ and the boundary 
conditions (\ref{qua33}) pointwise in $t$  we  get the following problem with respect to
$w^{k+1}$:
\begin{eqnarray}
	\partial_t w^{k+1} + A^k\partial_x w^{k+1} +B^{2}(x,t,u^k)w^{k+1}
	={\cal G}_{1}(k)w^{k+1}+{\cal G}_{2}(k)w^{k}+g^{3k}, \label{qua34}
\end{eqnarray}
\begin{equation}\label{qua35}
	\begin{array}{rcl}
		w_{j}^{k+1}(0,t)&=& R_jw^{k+1}(\cdot,t)+   h^{\prime\prime}_j(t)\\
		&&+ 
		\displaystyle\frac{d^2}{dt^2}\left[ \partial_2H_j(t,0)Q_j(t) u^k(\cdot,t)+ p_j^k(t)\right], \ 1\le j\le m,\\ [3mm]
		w_{j}^{k+1}(1,t)&=& R_jw^{k+1}(\cdot,t)+   h^{\prime\prime}_j(t)\\
		&&+ 
		\displaystyle\frac{d^2}{dt^2}\left[ \partial_2H_j(t,0)Q_j(t) u^k(\cdot,t)+ p_j^k(t)\right], \ m< j\le n,
	\end{array}
\end{equation}
where
\begin{eqnarray*} 
	B^{2}(x,t,u^k)&=&B^{1k}-\partial_2A^k\,(A^k)^{-1} = B^{k}-2\partial_2A^k\,(A^k)^{-1},\\ 
	 g^{3k}(x,t,v^k)&=&\Bigl( \partial_2g^{1k} - \partial_3 B^{1k}v^k  + 
	\left[\left(\partial_3 A^{k}v^k\right) (A^k)^{-1} + \partial_2 A^{k} (A^k)^{-1}\right](B^{1k} -  g^{1k}),\\
	&&- \partial_2B^{1} \Bigl) v^{k+1}  
	+ \partial_2g^{2k} - \left( \partial_3 A^{k}v^k \right)(A^k)^{-1} g^{2k} - \partial_2A^{k} (A^k)^{-1} g^{2k},\\
	\partial_3 B^{1}v^k &=&\left[\left\langle\nabla_uB_{ij}^1(x,t,u^k),v^k\right\rangle\right]_{i,j=1}^n,
\end{eqnarray*}  
and the operators
${\cal G}_{1}(k),{\cal G}_{2}(k)\in {\cal L}(C_{per}(\Pi;\bbbr^n))$
are defined by
\begin{equation}\label{1est}
	\begin{array}{rcl} 
	\left[{\cal G}_{1}(k)w^{k+1}\right](x,t) &=&\left[g^{1k} + 
	\left(\partial_3 A^k v^k\right)(A^k)^{-1}\right] w^{k+1},\\ [2mm]
	\left[{\cal G}_{2}(k)w^{k}\right](x,t)&=&\left(\partial_3 A^k w^k\right)(A^k)^{-1} v^{k+1}+\left(\partial_3 A^k\, w^k\right) (A^k)^{-1}(B^k u^{k+1} -f). 
\end{array} 
\end{equation}

On the account of (\ref{qua333}) and (\ref{1est}), the following estimate is true: 
\begin{eqnarray} \label{qua350}
	\| {\cal G}_2(k)w^k\|_{BC(\Pi;\bbbr^n)} \le L_5\delta_1\|w^k\|_{BC(\Pi;\bbbr^n)},
\end{eqnarray} 
where constant $L_5$ does not depend on $w^k.$ 
Because of (\ref{R'}), we have
\begin{equation}\label{2t}
\displaystyle  \frac{d^2}{dt^2}Q_j(t)u^k(\cdot,t)= \displaystyle Q_{0j}(t) u^k(\cdot,t) 
+Q_{1j}(t) v^k(\cdot,t) + Q_{2j}(t) w^k(\cdot,t).
\end{equation}

For every $j\le n$, decompose
$$
 \displaystyle\frac{d^2}{dt^2}\left[ \partial_2H_j(t,0)Q_j(t) u^k(\cdot,t)+ p_j^k(t)\right] 
	= {\cal H}_{1j}(k)  + {\cal H}_{2j}(k) w^k
$$
into two summands, where the second part ${\cal H}_{2j}(k)w^k$ consists of all summands with~$w^k$.
Taking into account (\ref{2t}), the summand $ {\cal H}_{2j}(k)w$ allows the representation
$$
{\cal H}_{2j}(k)w^k = \left[\partial_2H_j(t,0)  + 
2Q_j(t)	u^k(\cdot,t) \int_0^1
\d_2^2H_j(t,\sigma Q_j(t)u^k(\cdot,t))\dd\sigma\right]
Q_{2j}(t)w^k 
$$ 
and, therefore,  satisfies the estimate 
\begin{eqnarray} \label{qua355}
	\| {\cal H}_2(k)w^k\|_{BC(\Pi;\bbbr^n)} \le L_6(\varrho_1+2\delta_1\varrho_2\|Q\|)\|w^k\|_{BC(\Pi;\bbbr^n)}
\end{eqnarray} 
with a constant $L_6$ independent of $w^k.$ 

Evidently,
the function $w^{k+1}$ satisfies (\ref{qua34}) in a distributional sense
and (\ref{qua35}) pointwise if and only if it satisfies the following
operator equation obtained after the integration of (\ref{qua34}) along characteristic curves:
\begin{equation} \label{qua36}
	\begin{array}{rcl} 
		w^{k+1} &=& C(k)w^{k+1} + D(k) w^{k+1} 
		+ 		P(k)\left(h^{\prime\prime}+ {\cal H}_1(k)  + {\cal H}_2(k) w^k \right)\\ [2mm] &&
		+ S(k)\left({\cal G}_{1}(k)w^{k+1}+{\cal G}_{2}(k)w^{k}+g^{3k} \right),
	\end{array} 
\end{equation}
where the operators
 $C(k),D(k), S(k), P(k)\in {\cal L}(C_{per}(\Pi;\bbbr^n))$ are defined by 
\begin{eqnarray*} 
	\begin{array}{rcl}
			 \displaystyle [C(k)u]_j(x,t)&=& c_{jk}^2(x_j,x,t)R_j u(\cdot,\omega_j^k(x_j,x,t))),\nonumber
			 \\[2mm]
			 	[P(k)h]_j(x,t)&=& c_{jk}^2(x_j,x,t)h_j(\omega_j(x_j,x,t))),
		\\[2mm]
		  \displaystyle [D(k)u]_j(x,t)&=&\displaystyle 
		-\int_{x_j}^{x}d_{jk}^2(\xi,x,t)\sum_{i\neq j} b_{ji}^{2k}(\xi,\omega_j^k(\xi,x,t))u_i(\xi,\omega_j^k(\xi,x,t)) \dd\xi,\\[2mm]
			 \displaystyle [S(k)f]_j(x,t) &= &\displaystyle \int_{x_j}^{x}d_{jk}^2(\xi,x,t)f_j(\xi,\omega_j^k(\xi,x,t)) \dd\xi.
		\nonumber
	\end{array}
\end{eqnarray*}
Here the functions $b_{ji}^{2k}$ are elements of the matrix $B^{2k}$, while  the
functions $\omega_j^k, c_{jk}^l, d_{jk}^l$ are defined by (\ref{char}) and (\ref{cd}) with
$a_j$ and $b_{jj}$ replaced by $A_j^k$ and $b_{jj}^{k}$, respectively.
Note that $b_{jj}^{2k} = b_{jj}^{1k} - \partial_2 A_j^k (A_j^k)^{-1} = b_{jj}^{k} - 2 \partial_2 A_j^k (A_j^k)^{-1}.$ 
Analogously, we write $G_{0}(k), G_1(k),$ and $ G_2(k) $ for the operators $G_{0}, G_1,$ and $ G_2 $
defined by   (\ref{G_i}) where~$\omega_j$ is replaced by $\omega_j^k$.

After an iteration of (\ref{qua36}), we get the equation
\begin{eqnarray} \label{qua37}
	\begin{array}{rcl}
		w^{k+1} & = & C(k)w^{k+1} + \left(D(k) C(k)+ D^2(k)\right)w^{k+1} \\[2mm]
		&&+ (I+D(k)) S(k)\left({\cal G}_{1}(k)w^{k+1}+{\cal G}_{2}(k)w^{k}+g^{3k}\right)  \\[2mm]
		& & +(I+D(k))P(k)\left(h^{\prime\prime}+
		{\cal H}_1(k)  + {\cal H}_2(k) w^k\right),
	\end{array} 
\end{eqnarray}
Now we intend to show that for all sufficiently small $\delta\le \delta_1$
(and, hence for all sufficiently small $u^k$, see (\ref{qua5}), the last system can uniquelly  be solved with respect to $w^{k+1}$.
In other words, the system (\ref{qua37}) can be written in the following  
form:
\begin{equation} \label{oper5*}
	w^{k+1}  = {\cal A}(k)w^{k}+X^k, 
\end{equation}
where, for each $k$, we have ${\cal A}(k)\in{\cal L}(C_{per}(\Pi;\bbbr^n))$, $X^k\in C_{per}(\Pi;\bbbr^n)$, and
\begin{equation} \label{oper6*}
	\begin{array}{rcl}
		{\cal A}(k)w & =& 
		\left[I-C(k)-(I+D(k)) S(k){\cal G}_{1}(k)\right]^{-1} \\[2mm]
		&&\times(I+D(k)) S(k) {\cal G}_{2}(k)w
		+ (I+D(k)) P(k) {\cal H}_2(k) w,\\ [2mm]
		X^k&=&\left[I-C(k)-(I+D(k)) S(k){\cal G}_{1}(k)\right]^{-1} \Bigl(\left(D(k) C(k)+ D^2(k)\right)w^{k+1} \\[2mm]
		&&+ (I+D(k))S(k) g^{3k}
		+ (I+D(k))P(k) \left(h^{\prime\prime}+ {\cal H}_1(k)\right)\Bigl).
	\end{array} 
\end{equation}
To prove the equivalence of (\ref{qua37}) and  (\ref{oper5*}), it suffices to show that, for any $k\ge 1$,
the operator $I-C(k)-(I+D(k)) S(k){\cal G}_{1}(k)$ is invertable and, moreover, the inverse operator 
is bounded uniformly in $k \in \Z$. 
We first prove the invertibility of the operators  $I-C(k)$ in  $C_{per}(\Pi, \bbbr^n)$. To this end, we prove the
unique solvability of the system 
\begin{equation}\label{qua38}
	v_j(x,t)=c_{jk}^2(x_j,x,t)R_jv(\cdot,\omega_j^k(x_j,x,t))+g_j(x,t), \quad j\le n,
\end{equation}
in $C_{per}(\Pi,\bbbr^n)$ for any $g\in C_{per}(\Pi,\bbbr^n)$. 
Putting $x=0$ for $m<j\le n$ and $x=1$ for $1\le j\le m$ in (\ref{qua38}), 
we get the following system of $n$ equations with respect to  
$z(t) = (v_1(1,t),\dots, v_m(1,t), v_{m+1}(0,t),\dots,v_n(0,t))$:
\begin{equation}
	\label{1*}
	\begin{array}{rcl}
z_j(t)&=& \displaystyle c_{jk}^2(x_j,1 - x_j,t)\sum_{i=1}^nr_{ji}z_i(\omega_j^k(x_j,1 - x_j,t)) + g_j(1-x_j,t)\\
&=&[G_2(k) z]_j(t) + g_j(1-x_j,t), \quad  j\le n.
\end{array}
\end{equation}
Note that the system (\ref{qua38}) is uniquely solvable with respect to $v\in C_{per}(\Pi;\bbbr^n)$ if and only if  the system (\ref{1*}) is uniquely solvable with respect to $z\in C_{per}(\bbbr;\bbbr^n)$.

From the estimate
(\ref{qua5}) follows the bound
$\|G_2(k)\|_{{\cal L}(BC(\bbbr,\bbbr^n))}<c_0 < 1$, which uniform in $k\in\bbbn$ accordingly  to the proof of Theorem \ref{prop3}.
Hence,  the operator
$I-C(k)$ is invertible and, moreover, fulfills the estimate 
$$
\|(I-C(k))^{-1}\|_{{\cal L}(BC(\Pi;\bbbr^n))}\le 1+(1-c_0)^{-1}\|C(k)\|_{{\cal L}(BC(\Pi;\bbbr^n))}.
$$
Since the set of all invertible 
operators having bounded inverse is open, it suffices to show that
the operator $(I+D(k)) S(k){\cal G}_{1}(k)$ is  sufficiently small 
whenever $\delta_1$ is sufficiently small. As  Claim \ref{9}
is true with $\delta_2$
in place of $\delta_1$
for any $\delta_2\le\delta_1$,   for any $\sigma>0$ there is 
$\delta_2$ such that for all $u^k$ fulfilling
(\ref{qua5})  with $\delta_2$
in place of $\delta_1$, we have
$\|{\cal G}_1(k)\|_{{\cal L}(BC(\Pi;\bbbr^n))}\le\sigma$ for all $k$. 
Moreover,  the operators $D(k), P(k)$ and  
$S(k)$ are bounded uniformly in $k$.
Consequently, 
for any $f$ and $h$ satisfying (\ref{qua4}), the operator 
$I-C(k)-(I+D(k)) S(k){\cal G}_{1}(k)$
is invertible and the inverse is  bounded uniformly in $k$.
The desired conclusion
about the equivalence of (\ref{qua37}) and (\ref{oper5*}) follows.

To finish the proof of this claim,  it suffices to show that
the sequence $w^{k}$
converges in $C_{per}(\Pi;\bbbr^2)$ as $k\to\infty$. To this end,   we apply  to  the equation (\ref{oper5*}) a linear version of the 
fiber contraction principle (see \cite[Section 1.11.3]{Chicone}  ) or,  more precisely, a
  version of  it given by  \cite[Lemma A.1]{Hopf} and written in terms of our problem as follows.

\begin{lemma} \label{fiber}
	Let the following conditions be fulfilled:
	
	$(i)$
	$
	{\cal A}(k)u$ converges in $ C_{per}(\Pi;\bbbr^n)$ as $ k\to\infty
	$
	for all $u\in C_{per}(\Pi;\bbbr^2)$,
	
	$(ii)$
	there exists $c<1$  such that 
	$
	\|{\cal A}(k)\|_{{\cal L}(BC(\Pi;\bbbr^2))}\le c,
	$
	for all $k\in \bbbn$,
	
	$(iii)$
	$
	X^k \mbox{ converges in } C_{per}(\Pi;\bbbr^n) \mbox{ as } k\to\infty.
	$
	\\
	Then the sequence $w^{k}$ given by (\ref{oper5*})
	converges in $C_{per}(\Pi;\bbbr^n)$ as $k\to\infty$.
\end{lemma}

To varify the conditions of this lemma, we follow  the same line as in \cite[p. 4204--4205]{jee}.
Thus, since all the operators in the right-hand side of first formula of $(\ref{oper6*})$ do not depend on $w^{k}$ for all $k$, the condition $(i)$ of the lemma easily
follows from Claim~\ref{9}. 

By (\ref{qua350}) and  (\ref{qua355}),
for any $\sigma>0$
there exist $\delta\le\delta_2$ and $\varrho > 0$ 
such that for all $f$ and $h$ satisfying the bound (\ref{qua4}) 
with $\delta$ in place of $\delta_1$ 
 it holds
$\|{\cal G}_2(k)\|_{{\cal L}(BC(\Pi;\bbbr^n))}  + \|{\cal H}_2(k)\|_{{\cal L}(BC(\Pi;\bbbr^n))} \le \sigma$.
Hence, taking into the account that all other operators in the right hand side of 
the first equality in (\ref{oper6*}) are bounded uniformly in $k$ 
whenever $\delta$  is sufficiently small, we state that, for sufficiently small  $\delta$ and  $\varrho_1 $, the
condition $(ii)$ is  fulfilled for all $f$ and $h$ satisfying the bound (\ref{qua4}) with $\delta$ in place of $\delta_1.$

For the remaining property $(iii)$, note that the operators 
$D(k)$ and $C(k)$  depend  neither on $w^k$ nor on 
$v^k$ for all $k$.  Moreover, the operators $D(k) C(k)$ and $D(k)^2$ are smoothing mapping $BC(\Pi;\bbbr^n))$
into $BC^1_t(\Pi;\bbbr^n))$. This entails that the expressions $D(k) C(k)w^{k+1}$ and $D(k)^2 w^{k+1}$, actually, 
do not depend on $w^{k+1}$ but only on $v^{k+1}$.
 Moreover, analogously to (\ref{ots3}),
 the following  estimate is true:
\begin{eqnarray*} 
	\left\|\partial_t\left(D(k)C(k) + D(k)^2\right) w^{k+1}\right\|_{BC(\Pi;\bbbr^n)} \le L_{7} \|v^{k+1}\|_{BC(\Pi;\bbbr^n)}
\end{eqnarray*}
for some $L_{7}$ not depending on $k$.
It follows  that the right hand side of the second formula in (\ref{oper6*}) does not depend on $w^{k}$ for all $k$ and, therefore, the convergence of $X^k$ immediately follows
from Claim \ref{9}.  

Set 
$$\varepsilon =\min\left\{\frac{\delta}{L_4}\left[1-\left(\varrho_1 +  \delta\varrho_2\|Q\|\right)\|Q\|L_4\right]
, \frac{(1-\varrho_1)\left[1 -(\varrho_1 +  \varrho_2\delta)\|Q\|L_4\right]}{L_3L_4(N_1 + N_2)}\right\}.$$
Therefore, 
by assumptions of Claim \ref{10} the sequence $w^{k}$
converges in $C_{per}(\Pi;\bbbr^n)$ as $k\to\infty$.

Now, by  Claim \ref{9}, we conclude that the second derivative in $t$
of $u^*$ exists and that the sequence $\partial^2_tu^k$ converges to $\partial^2_tu^*$
in $C_{per}(\Pi;\bbbr^n)$
as $k\to\infty$.
Differentiating (\ref{qua1}) first in $t$ and then in $x$, we conclude that $u^k$ converges to $u^*$  in $C_{per}^2(\Pi;\bbbr^n)$ as $k\to\infty$. It follows also that the estimate (\ref{*1**}) is fulfilled.
\end{subproof}

\begin{claim}\label{11} Let the conditions  of Claim \ref{10} are fulfilled.
	Then the bounded classical solution $u^*$  to the problem  (\ref{1}), (\ref{2}) 
	is  unique.
\end{claim}
\begin{subproof}
Assume,   contrary to the claim, that $\tilde u$ is a classical  solution to the problem 
(\ref{1}), (\ref{2}) which satisfies  the estimate  (\ref{*1**}) and  differs from $u^*$. 
Then functions $\widetilde A(x,t) = A(x,t, \tilde u(x,t))$ and  $\widetilde B(x,t) = B(x,t, \tilde u(x,t))$
satisfy inequalities 
$$\|\widetilde A - A^0\|_{BC^2(\Pi,\bbbr^n)} \le \gamma_1, \quad
\|\widetilde B- B^0\|_{BC^2(\Pi,\bbbr^n)} \le \gamma_1.$$

The difference $\tilde y^{k+1}= u^{k+1}-\tilde u$ belongs to $C^2_{per}(\Pi;\bbbr^n)$  
and satisfies the  system
$$ \partial_ t y^{k+1} + A(x,t,\tilde u) \partial_ xy^{k+1}+ B(x,t,\tilde u) y^{k+1} 
= \tilde f^k(x,t)$$
and the boundary conditions  
\begin{eqnarray*}
	& &  y^{k+1}_{j}(0,t)= R_j y^{k+1}(\cdot,t) +  H_j(t, Q_ju^{k+1}(\cdot,t)) - H_j(t, Q_j\tilde u(\cdot,t)), \
	1\le j\le m,\\ 
	& & y^{k+1}_{j}(1,t)= R_jy^{k+1}(\cdot,t) + H_j(t, Q_ju^{k+1}(\cdot,t)) - H_j(t, Q_j\tilde u(\cdot,t)), \ m< j\le n,
\end{eqnarray*}
for all $\ t\in \bbbr$, where 
$$\tilde f^k(x,t)   =  \left(\widetilde B - B^{k}\right)u^{k+1} + 
\left(\widetilde A - A^{k}\right)\partial_ xu^{k+1}.$$

Analogously to (\ref{qua23}), we derive  the inequality
\begin{eqnarray}\label{qua23u}
	\begin{array}{ll}
		\|y^{k+1}\|_{BC^1(\Pi;\bbbr^n)} \le L_3\left(\| \tilde f^k \|_{BC_t^1(\Pi;\bbbr^n)} + 
		\|  H^{k+1}-\widetilde H \|_{BC^1(\bbbr;\bbbr^n)} \right)  \\[2mm]
\quad	 \qquad	\le\displaystyle
		\displaystyle\|y^k\|_{BC^1(\Pi;\bbbr^n)}\biggl(\varrho_1+\frac{L_3L_4(N_1 + N_2)}{1 -(\varrho_1 +  \varrho_2\delta_1)\|Q\|L_4}  \\[4mm]
	\quad	\qquad\times\left( \| f \|_{BC^1(\Pi;\bbbr^n)} + \|\partial_t^2 f \|_{BC(\Pi;\bbbr^n)}  + \|  h\|_{BC^2(\bbbr;\bbbr^n)} \right)
		\biggr),
	\end{array}
\end{eqnarray}
where $\widetilde H_j(t)=H_j(t, Q_j\tilde u(\cdot,t))$ for all $j\le n$.
If $ \| f \|_{BC^1(\Pi;\bbbr^n)} + \|\partial_t^2 f \|_{BC(\Pi;\bbbr^n)}  + \| h\|_{BC^2(\bbbr;\bbbr^n)}< \varepsilon_1,$ where $\varepsilon_1$ is defined by (\ref{qua24}),
then,  by the inequality (\ref{qua23u}), the sequence $y^k$ is contracting and, hence  tends to zero in  $BC^1(\Pi;\bbbr^n).$	 This means  that $\tilde u = u^*,$ and this is the desired contradiction.  
\end{subproof}
The proof of Theorem \ref{main} is complete.

\end{document}